\documentclass[11pt]{article}
\usepackage{amsmath,amssymb,amsthm,graphicx,subfigure,float,url}

\usepackage{tikz}

\usepackage{parskip} 
\usepackage{pdfsync}
\graphicspath{{fig/}}

\usepackage[colorlinks=true]{hyperref} 
\usepackage{pdfsync}

\usepackage{enumitem}
\usepackage{dsfont} 
\usepackage[T1]{fontenc} 

\def\R{\textrm{I\kern-0.21emR}}
\def\N{\textrm{I\kern-0.21emN}}

\renewcommand{\geq}{\geqslant}
\renewcommand{\leq}{\leqslant}

\renewcommand{\geq}{\geqslant}
\renewcommand{\leq}{\leqslant}

\newcommand {\de} {\delta}

\newcommand {\Chi} {{\bf \raise 2pt \hbox{$\chi$}} }
\newcommand {\f}   {\frac}

\newcommand {\ddt}   { \frac{d}{dt}}  

\newcommand  \ind[1]  {   \mathds{1}_{#1}   }
\newcommand{\re}{\eqref}

\newcommand{\la}{\label}

\newcommand{\beq}{\begin{equation}}
\newcommand{\eeq}{\end{equation}}
\newcommand{\bea} {\begin{array}{rl}}
\newcommand{\eea} {\end{array}}
\newcommand{\bepa}{\left\{ \begin{array}{l}}
\newcommand{\eepa} {\end{array}\right.}
\DeclareMathOperator*{\amax}{arg\,max}
\newcommand{\bmu}{\begin{multline}}
\newcommand{\emu}{\end{multline}}

\newtheorem{theorem}{Theorem}  
\newtheorem{proposition}{Proposition}
\newtheorem{corollary}{Corollary}

\newtheorem{lemma}{Lemma}

\theoremstyle{definition}\newtheorem{remark}{Remark}

\title{Asymptotic analysis and optimal control of an integro-differential system modelling healthy and cancer cells exposed to chemotherapy}

\author{Camille Pouchol\thanks{Sorbonne Universit\'es, UPMC Univ Paris 06, CNRS UMR 7598, Laboratoire Jacques-Louis Lions, F-75005, Paris, France} \thanks{INRIA Team Mamba, INRIA Paris, 2 rue Simone Iff, CS 42112, 75589 Paris, France} \footnotemark[4]\; \, Jean Clairambault\footnotemark[2] \footnotemark[1] \footnotemark[4]\; \, Alexander Lorz\thanks{CEMSE Division, King Abdullah University of Science and Technology, Thuwal 23955-6900, Saudi Arabia} \footnotemark[1] \footnotemark[2] \footnotemark[4] \; Emmanuel Tr\'elat\footnotemark[1] \thanks{e-mail: \href{mailto:pouchol@ljll.math.upmc.fr}{pouchol@ljll.math.upmc.fr} (corresponding author), \href{mailto:jean.clairambault@inria.fr}{jean.clairambault@inria.fr}, \href{mailto:alexander.lorz.1@kaust.edu.sa}{alexander.lorz.1@kaust.edu.sa}, \; \href{mailto:emmanuel.trelat@upmc.fr}{emmanuel.trelat@upmc.fr}}}

\begin{document}

\newcounter{assum}

\maketitle

\begin{abstract}
We consider a system of two coupled integro-differential equations modelling populations of healthy and cancer cells under therapy. Both populations are structured by a phenotypic variable, representing their level of resistance to the treatment. We analyse the asymptotic behaviour of the model under constant infusion of drugs. By designing an appropriate Lyapunov function, we prove that both densities converge to Dirac masses.
We then define an optimal control problem, by considering all possible infusion protocols and minimising the number of cancer cells over a prescribed time frame. We provide a quasi-optimal strategy and prove that it solves this problem for large final times. For this modelling framework, we illustrate our results with numerical simulations, and compare our optimal strategy with periodic treatment schedules.
\end{abstract}

\section{Introduction}
\label{Section1}
One of the primary causes of death worldwide is cancer~\cite{jemal2016}. Cancer treatment encounters two main pitfalls: the emergence of drug resistance in cancer cells and toxic side effects to healthy cells. Given these causes of treatment failure, designing optimized therapeutic strategies  is a major objective for oncologists. In this paper, we propose a mathematical framework for modelling these phenomena and optimally combining therapies. 

\subsection{Overview and motivation} 

The most frequently used class of anti-cancer drugs are chemotherapeutic (cytotoxic) drugs, which are toxic to cells, leading to cell death. For example, platinum-based agents kill dividing cells by causing DNA damage and disrupting DNA replication~\cite{Kelland2007}. Another class of drugs are cytostatic drugs, which slow down cell proliferation without killing cells. For example, trastuzumab is a cytostatic drug used in breast cancer treatment that targets growth factor receptors present on the surface of cells, and inhibits their proliferation~\cite{Gabriel2005}. Despite this obvious functional difference between the two classes, cytostatic drugs, such as tyrosine kinase inhibitors, can also be cytotoxic at high doses~\cite{Rixe2007}.  

It is a well documented fact that cytotoxic agents can fail to control cancer growth and relapse~\cite{Hanahan2000, Pasquier2010, Scharovsky2009}.  First, eradication of the tumour cell population is compromised by the emergence of drug resistance, due to intrinsic or acquired genotypic and phenotypic heterogeneity in the cancer cell population~\cite{bedard2013, burrell2013, greaves2015, navin2010}, because a subpopulation of resistant cells survives and proliferates, even in the presence of further treatment with identical~\cite{sharma2010}, or higher doses~\cite{pisco2015}. Second, chemotherapeutic treatments have unwanted side effects on healthy cells, which precludes unconstrained treatment use for fear of unwanted toxicities to major organs. It is therefore a challenge for oncologists to optimally and safely treat patients with chemotherapy. \par

The medical objective of killing cancer cells together with preserving healthy cells from excessive toxicity is routinely translated in mathematical terms as finding the best therapeutic strategies (\textit{i.e.}, below some maximum tolerated dose, referred to as MTD) in order to minimise an appropriately chosen cost function. There are many works in mathematical oncology focusing on the optimal modulation of chemotherapeutic doses and schedules designed to control cancer growth, {\it e.g.}~\cite{agur2006, costa1992, costa1994, kimmelswierniak2006, ledzewicz2006, ledzewicz2007a, ledzewiczschattler2008, ledzewiczMBE2011, swan1977, swanbook1984, swan1990}. \par

Since using ordinary differential equations (ODEs) is a common technique for modelling the temporal dynamics of cell populations, the mathematical field of optimal control applied to ODEs has emerged as an important tool to tackle such questions (see for instance~\cite{Schaettler2015} for a complete presentation).
In these ODE models, toxicity can either be incorporated in the cost functional as in~\cite{costa1994}, or by adding the dynamics of the healthy cells~\cite{billy2013c}. One simple, but rather coarse, paradigm used to represent drug resistance in such ODE models is by distinguishing between sensitive and resistant cancer cell subpopulations~\cite{costa1994,ledzewicz2006}. Herein, the main tools available to obtain rigorous results are the Pontryagin maximum principle (PMP) and geometric optimal control techniques~\cite{Agrachev2004, Pontryagin1964,Schaettler2012,Trelat2008}. \par

Another paradigm used in the mathematical modelling of drug resistance relies on the idea that \textit{phenotypic heterogeneity} in cancer cells and the dynamics of cancer cell populations can be understood through the principles of Darwinian evolution~\cite{greaves2010, greaves2012}.  Given a particular tumor micro- and macro-environment (\textit{e.g.}, access to oxygen, nutrients, growth factors, drug exposure), the fittest cells are selected. In the case of resistance, resistant cell subpopulations are assumed to emerge and be selected for their high levels of fitness in the presence of chemotherapeutic agents. Whether they already exist in the cell population, surviving and remaining dormant at clinically undetectable small numbers, and emerging only by natural selection, or they do not exist at all initially, but emerge as a result of an evolutionary trade-off between proliferation and development of costly survival mechanisms~\cite{aktipis2013}, likely of epigenetic nature, is still difficult to decide. The two scenarios have been studied in a modelling framework in~\cite{Chisholm2015}. \par

Adaptive dynamics is a branch of mathematical biology that aims at modelling Darwinian selection~\cite{Diekmann2004, Diekmann2005}. It is thus a natural theoretical framework for the representation of phenotypic evolution in proliferating cell populations exposed to anti-cancer drugs and tumor micro-environmental factors. Non-Darwinian evolutionary principles have also been proposed to take into account drug resistance phenomena~\cite{pisco2013}. Adaptive dynamics is amenable to modelling these principles as well. To this end, stochastic or game-theoretic points of view (see~\cite{Champagnat2008, Hofbauer1990}) are standard in adaptive dynamics. Apart from ODEs, partial differential equations (PDEs) and integro-differential equations (IDEs) represent other deterministic approaches. The latter ones represent the focus of our paper. For an introduction to PDE and IDE models in adaptive dynamics, we refer the interested reader to~\cite{Perthame2006,Lorz2011}. \par 

A common feature of these modelling techniques is that the population is structured by a trait, referred to as \textit{phenotype}. The resistance level of a cell to a drug therapy is an example of such trait. Often, this variable is assumed to be \textit{continuous} since it is correlated with biological characteristics, \textit{e.g}., the intracellular concentration of a detoxication molecule (such as reduced glutathione), the activity of detoxifying enzymes in metabolising the administered drug, or drug efflux transporters eliminating the drug. Another possible continuous structuring variable is the ability of some cancer cells to quickly change their phenotypes (otherwise said, their intrinsic {\it plasticity}) by regulating the level of DNA methylation and/or of activity of DNA methyltransferases~\cite{Chisholm2016, sandoval2012}. This ability is also correlated with the degree of resistance to a given drug. \par

To this end, a relevant modelling alternative to the binary sensitive versus resistant ODE framework (as already proposed long ago in e.g.,~\cite{costa1992, costa1994}) consists of studying the cells at the population level using structured population dynamics. Specifically, let us denote the density of cells at time $t$ and with phenotype $x$ by $n(t,x)$, with $x\in{[0,1]}$. The continuous phenotype $x$ represents an abstract level of resistance (which may be molecularly related to the activity level of an ABC transporter, or to a mean level of methylation of the DNA) to a cytotoxic drug in a cell population. Such models allow for the analysis of the asymptotic behaviour in terms of an asymptotically selected  phenotype  and of the total population $\rho(t) := \int_0^1 n(t,x) \, dx$. In the classical non-local logistic model, written as the IDE
\begin{equation*}
     \dfrac{\partial n}{\partial t} (t,x) 
        							=  \big(r(x) - d(x) \rho(t) \big) n(t,x),
\end{equation*}
where $\rho(t) := \int_0^1 n(t,x) \, dx$, cells proliferate at rate $r(x)$ and die at rate $d(x) \rho(t)$ (the more individuals, the more competition and thus death). Such equations have well known asymptotic properties, such as \textit{convergence} and \textit{concentration}~\cite{Perthame2006, Desvillettes2008, Jabin2011}. For large times,  $\rho$ converges to the smallest value $\rho^\infty$ such that $r - d \rho^\infty\leq 0$ on $[0,1]$ and $n(t,\cdot)$ concentrates on the set of points such that  $r(x) -d(x) \rho^\infty=0$. The limit is thus typically expected to be a sum of Dirac masses. This phenomenon can be interpreted as the \textit{selection} of dominant traits by the environment. 
A common strategy for proving this asymptotic behaviour consists in showing that $\rho$ has a bounded variation ($BV$) on $[0,+\infty)$, as in~\cite{Lorz2011, Lorz2014}. \par 

To the best of our knowledge, general results of convergence and concentration are still elusive for systems of IDEs: the methods used in the $BV$ framework do not seem to generalise, even in the setting that is of special interest to us, namely in the case of two competitively interacting populations of (healthy and cancer) cells. This leads to the asymptotic analysis of systems of the form
\begin{equation}
\label{2x2}
\begin{split}
     \dfrac{\partial n_1}{\partial t} (t,x)&= \big(r_1(x) - d_1(x) I_1(t) \big) n_1(t,x), \\
     \dfrac{\partial n_2}{\partial t} (t,x)&= \big(r_2(x) - d_2(x) I_2(t) \big) n_2(t,x).
\end{split}
\end{equation}
The competitive coupling comes from $I_1=a_{11} \rho_1 + a_{12}\rho_2$, $I_2=a_{22} \rho_2 + a_{21}\rho_1$ with $\rho_i(t)= \int_0^1 n_i(t,x) \, dx$, $i=1,2$. In particular, it is not clear \textit{a priori} whether such interactions may or may not lead to \textit{oscillatory} behaviours at the level of $\rho_1$, $\rho_2$. We mention the work~\cite{Busse2016} where convergence and concentration are completely characterized for a 2x2 system, where a triangular coupling structure is considered. Due to an appropriately designed Lyapunov function, the results of our paper imply that convergence and concentration hold for the model \eqref{2x2}. \par
Our goal here is also to include control terms in order to model the effect of the drugs on the proliferation and death rates. The equations we will use henceforth throughout this paper are
\begin{equation}
\label{2x2Control}
\begin{split}
     \dfrac{\partial n_1}{\partial t} (t,x)&= \left(\frac{r_1(x)} {1+\alpha_1 v(t)} - d_1(x) I_1(t)-\mu_1(x) u(t) \right) n_1(t,x),  \\
     \dfrac{\partial n_2}{\partial t} (t,x)&= \left(\frac{r_2(x)} {1+\alpha_2 v(t)} - d_2(x) I_2(t)-\mu_2(x) u(t) \right) n_2(t,x).  
\end{split}
\end{equation}
Here the asymptotic analysis is more complex because of the controls $u$ and $v$. On a fixed time-frame $(0,T)$, we will search among controls $u$, $v$ in $BV(0,T)$, since it would not biologically feasible to impose fast varying drug infusion rates to patients.
\par
The main difference between the approach developed in this paper and the ODE ones is that the model considered throughout is an IDE approach studied from an optimal control perspective. This model is inspired by~\cite{Lorz2013a, Lorz2015} where an IDE model (possibly with an additional space variable) has been used in order to model the effect of constant doses of cytotoxic and cytostatic drug chemotherapies. \par
However, a mathematical proof of the failure of chemotherapeutic treatments at MTD levels to eradicate cancer cell population has not yet been formulated in the IDE modelling framework, a result that we will obtain in the present study. In addition, previous works have only considered a priori prescribed drug treatment schedules. \par 

In this paper, we show that our model (\ref{2x2Control}) is consistent with clinical observations on the effect of constant infusion of high doses, and we address the optimal control problem of such IDE models. Our study has a potential impact for oncologists and mathematical biologists, since it provides an accurate and robust understanding of possible optimal strategies. Up to our knowledge, this is the first time that a mathematical model and its optimal control reflect the emerging fact, observed and acknowledged recently by many clinicians, that giving maximal tolerated drug doses, even periodically, may finally be detrimental and in any case is far from being an optimal strategy in view of curing cancer.

\subsection{Modelling and overview of the main results}
We consider both the healthy and cancer cell populations, modelled by their respective densities $n_H(t,x)$ and $n_C(t,x)$, where the variable $x\in{[0,1]}$, called \textit{phenotype}, represents drug resistance levels: a cell of phenotype $x$ is highly sensitive if $x$ is close to $0$, and is highly resistant if $x$ is close to $1$. Chemotherapy is modelled by two functions of time $u_1$ and $u_2$, representing cytotoxic and cytostatic drug infusion flows, respectively, in the two cell populations. 
These functions are the controls, assumed to be bounded variation functions of time if the equation is set on a bounded time-frame, subject to maximum tolerated doses (MTD) thresholds:
\begin{equation}\label{u_max}
0\leq u_1(t)\leq u_1^{\max},\qquad 0\leq u_2(t)\leq u_2^{\max}.
\end{equation}
We assume that the densities $n_H$ and $n_C$ satisfy the following lDE system:
\begin{equation}\label{contsyst}
\begin{split}
\dfrac{\partial n_H}{\partial t} (t,x) &= R_H\left(x, \rho_H(t), \rho_C(t),u_1(t),u_2(t)\right)
n_H(t,x), \\ \vspace{.8em}
\dfrac{\partial n_C}{\partial t} (t,x)& = 
R_C\left(x, \rho_C(t), \rho_H(t), u_1(t), u_2(t)\right) n_C(t,x), 
\end{split}
\end{equation}
with the net growth rates defined as 
\begin{equation}\label{FunctionsR}
\begin{split}
R_H(x,\rho_H,\rho_C,u_1,u_2) &:= \frac{r_H(x)}{1+\alpha_H u_2} - d_H(x) I_H - u_1 \mu_H(x), \\ \vspace{.8em}
R_C(x,\rho_C, \rho_H, u_1,u_2) &:= \frac{r_C(x)}{1+\alpha_C u_2} - d_C(x) I_C - u_1 \mu_C(x), \vspace{.8em}
\end{split}
\end{equation}
the non-local coupling as
\begin{equation}
\label{Logistic}
I_H :=  a_{HH} \rho_H + a_{HC} \rho_C, \; \; I_C := a_{CH} \rho_H + a_{CC} \rho_C,
\end{equation}
with
\begin{equation*}
\rho_H(t) = \int_0^1 n_H(t,x)\, dx,\; \;
\rho_C(t) = \int_0^1 n_C(t,x)\, dx,
\end{equation*}
which are the total number of healthy and tumour cells at time $t$. \par 
The system starts from the initial conditions
\begin{equation}\label{initial}
n_H(0,x) = n_H^0(x), \; \; 
n_C(0,x) = n_C^0(x).
\end{equation}
\medskip
In the above model: \par
\noindent
$\bullet$ \,
$r_H$ and $r_C$ are the drug-free proliferation rates, assumed to be positive, decreasing functions on $[0,1]$. \\
$\bullet$ \,
The factors $\f{1}{1+\alpha_H u_2(t)}$ and $\f{1}{1+\alpha_C u_2(t)}$ model the decrease in proliferation rates due to cytostatic drugs. The positive constants $\alpha_H$ and $\alpha_C$  represent average sensitivities of cells to cytostatic drugs. Throughout, we make the assumption that cancer cells are more sensitive to the  drugs, \textit{i.e.},
\begin{equation}
\begin{split}
\label{AssumptionAlpha}
\alpha_H < \alpha _C.
\end{split}
\end{equation}
$\bullet$ \,
The terms $d_H\, I_H$ and $d_C\, I_C$ are the drug-free death rates. The functions $d_H$ and $d_C$ are positive, decreasing functions on $[0,1]$. Given the dependence in $I_H$ and $I_C$, the model resembles a logistic one. According to their definitions, these functions  are linear combinations of the total population $\rho_H$ and $\rho_C$, \textit{i.e.}, we consider both intraspecific and interspecific competition. We assume that the intraspecific competition is stronger than the interspecific one:
\begin{equation}
\label{Competition}
0<a_{HC}<a_{HH}, \; \;
0<a_{CH}<a_{CC}.
\end{equation}
$\bullet$ \,
The terms $u_1(t) \mu_H(x)$ and $u_1(t) \mu_C(x)$ are additional death rates due to cytotoxic drugs, with $\mu_H$ and $\mu_C$ assumed to be non-negative, decreasing functions on $[0,1]$. These functions may vanish on some interval $[1-\epsilon, 1]$, which in this case reflects the fact that some cells become fully resistant to those drugs. \par
$\bullet$ \,
This model imposes that the phenotype $x$ be mostly linked to the resistance to cytotoxic drugs, in accordance with ~\cite{pisco2015}. This fact which will be made clearer by the analysis of the model (see below among the consequences of Theorem \ref{thm_asympt}). 
\paragraph{Asymptotic behaviour for controls in $BV\left([0,+\infty)\right)$.} 
Our first aim is to show that our model reproduces the following clinical observations: when high drug doses are administered, the tumour first reduces in size before regrowing, insensitive to further treatment. 

We thus want to establish asymptotic properties of the model, a challenging task since it is similar to system (\ref{2x2Control}). The following statement is our first main result: we achieve a complete description of the asymptotic behaviour of system (\ref{contsyst}), with a class of asymptotically constant controls. 

\smallskip
\begin{theorem}
\label{thm_asympt}
Let $u_1, u_2$ be any functions in $BV\left([0,+\infty)\right)$, and let $\bar u_1$, $\bar u_2$ be their limits at $+\infty$. Then, for any positive initial population of healthy and of tumour cells, $\left(\rho_H(t),\rho_C(t)\right)$ converges to some equilibrium point $(\rho_H^\infty,\rho_C^\infty)$, which can be explicitly computed. \par

Furthermore, $n_H$ and $n_C$ concentrate on a set of points which can also be explicitly computed.
\end{theorem}
The explicit values can be found in Section \ref{Section2}, where this result is proved. If $\bar u_1=0$, the sets of points on which $n_H$ and $n_C$ concentrate are independent of $\bar u_2$. This is due to the fact that the phenotypic variable $x$ models resistance to cytotoxic drugs. \par
Thus, system (\ref{contsyst}) also exhibits convergence and concentration: the classical features of a single IDE generalise to a system, and our method is flexible enough to incorporate controls in $BV([0,+\infty))$. The main ingredient of the proof is an appropriately designed Lyapunov function, inspired by~\cite{Jabin2011} and by the classical Lyapunov functions used for studying the global asymptotical stability of steady-states in Lotka-Volterra ODE systems, as in~\cite{Goh1977}. The proof then consists of a fine analysis of this Lyapunov functional, yielding both convergence and concentration, while providing estimates on their speed. This method is new in the analysis of such IDE systems. \par 
If $\mu_C$ vanishes identically on some interval $[1-\varepsilon,1]$ (meaning that full resistance is possible), this theorem explains why, in the long run, high doses are not optimal. This means that our mathematical conclusions are in agreement with the idea that the standard method used in the clinic, namely administering maximum tolerated doses, should be reconsidered. Alternatives are currently extensively being investigated by oncologists, e.g., metronomic scheduling, which relies on frequent and continuous low doses of chemotherapy~\cite{Benzekry2013, Carrere2017, Pasquier2010}. \par 

Theorem \ref{thm_asympt} thus motivates the optimal control problem of searching for the best possible functions $u_1$ and $u_2$ to minimise the number of cancer cells within a given horizon of time, which we now introduce in more details. \par

\paragraph{Optimal control problem: optimal chemotherapy strategy.} We fix some $T>0$, and assume that the initial conditions $n_H^0$ and $n_C^0$ are continuous and positive functions on $[0,1]$. For any $(u_1, u_2)$ in $\left(BV(0,T)\right)^2$ which satisfy \re{u_max}, we consider the associated trajectory $(n_H(\cdot,x), n_C(\cdot,x))$ on $[0,T]$, solution of the system \re{contsyst} starting from $(n_H^0, n_C^0)$. 
We also take into account two state constraints:
\par 
$\bullet$ \,
it is required to keep a minimal proportion of healthy cells with respect to the total number of cells, and hence we impose that
\begin{equation}\label{cont_HC}
\frac{\rho_H(t)}{\rho_H(t)+\rho_C(t)} \geq \theta_{HC},
\end{equation}
for some $0<\theta_{HC}<1$. \par
$\bullet$ \,
Moreover, we impose that the number of healthy cells always remains above a certain fraction of the initial number of healthy cells:
\begin{equation}
\label{cont_H}
\rho_H(t) \ge \theta_H \rho_H(0),
\end{equation}
for some $0<\theta_{H}<1$.  \par
We define $\cal{A}$$_T$ as the set of admissible controls, \textit{i.e.}, for which those constraints are satisfied on $(0,T)$. For given $(u_1, u_2)\in{\cal{A}}_T$, we define the associated cost as the number of cancer cells at the end of the time-frame:
\begin{equation}
\label{Cost}
C_T(u_1,u_2):=\rho_C(T).
\end{equation}


\noindent
We now define the optimal control problem, denoted in short ({\bf OCP}) in the sequel, as 
\begin{equation}
\label{OCP}
\inf_{(u_1, u_2)\in{\cal{A}}_T} C_T(u_1,u_2).
\end{equation}
In other words, we want to find the best drug administration strategy to minimise the number of cancer cells at the end of a fixed time-frame $[0,T]$, while both keeping toxicity to a tolerable level and controlling tumor size. It might seem more natural to study the problem in free final time, but as explained later on, the mapping $T\longmapsto \rho_C(T)$ (where $\rho_C(T)$ is the optimal value obtained by solving ({\bf OCP}) on $[0,T]$) is decreasing in $T$. This implies that the optimal control problem in free final time $T$ is ill-posed and does not admit any solution. The other implication is that when solving the optimal control problem in free final time $t_f$ under the constraint $t_f\leq T$ (where $T$ is a horizon), then the optimal solution will be such that $t_f=T$. This is why we focus on an optimal control problem in fixed final time.\par

In this paper, we perform a thorough study of ({\bf OCP}), both theoretically and numerically. 

The theoretical study is made on a smaller class of controls which, after a long phase of constant doses, are allowed to vary on a small final time frame. More precisely, for a given $T_1<T$, we consider the subclass $\mathcal{B}_T \subset \mathcal{A}_T$ defined by
$$\mathcal{B}_T:= \bigg\{(u_1, u_2) \in \mathcal{A}_T, \; (u_1(t),u_2(t)) = (\bar u_1, \bar u_2) \; \text{on} \; (0,T_1), \; T-T_1 \leq T_2^{M} \bigg\}$$ where $T$ is large and where the optimal length of the second phase $T_2:=T-T_1$ is bounded above by some small constant $T_2^{M}$. Thus, the first phase is long. Optimising within this class is equivalent to searching for constant optimal values $\bar u_1$, $\bar u_2$ of the controls during the first phase, the length of the second phase $T_2 \leq T_2^{M}$, and optimal $BV(T_1,T)$ controls $u_1$, $u_2$ on $(T_1,T)$. The reason for this restriction to this class of controls comes from the answer to the following question: given a specific tumour size (\textit{i.e.}, a given number of cancer cells), what would be the optimal phenotypic cellular distribution in order tumor burden at the end of the time interval? Proposition \ref{DiracStart} shows that, for a very short time, it is always better that the cancer cell population be concentrated on some appropriate phenotype, \textit{i.e.}, that the initial population be a Dirac mass at some appropriate point. 
\par 
From Theorem \ref{thm_asympt}, we know that it is possible to asymptotically reach Dirac masses with constant controls. The combination of these two results justifies the analysis in $\mathcal{B}_T$. 
\par
In this class of controls, our second main result characterises a quasi-optimal strategy in large time, a result which we now state informally. \par 
\textit{Consider the problem of minimising $\rho_C(T)$ within the class $\mathcal{B}_T$. 
When $T$ is large enough, the optimal strategy approximately consists of:}
\begin{itemize}
\item[$\bullet$] \textit{a first long-time arc, with constant controls on $[0,T_1]$, at the end of which populations have almost concentrated in phenotype (for $T$ large);
\item[$\bullet$] a last short-time part, on $[T_1,T]$ consisting of at most three arcs (for small $T_2 = T-T_1$):} 
\begin{itemize}
\item[-] \textit{a boundary arc\footnote{A boundary arc (for the state constraint $g\left(\rho_H,\rho_C\right) \leq 0$) is an arc along which $g\left(\rho_H,\rho_C\right)=0$, i.e., the constraint is saturated. A free arc is an arc along which $g\left(\rho_H,\rho_C\right) < 0$, i.e., the constraint is not saturated.}
\item[-] a free arc with controls $u_1= u_1^{max}$ and $u_2= u_2^{max}$;}
\item[-] \textit{a boundary arc along the constraint (\ref{cont_H}) with $u_2= u_2^{max}$. }
\end{itemize}
\end{itemize}
The precise result and hypotheses are given by Theorem \ref{SolveOCP} in Section \ref{Section3}. 

In order for Theorem \ref{SolveOCP} to hold, an important assumption we make is that when cancer cells are concentrated on a sensitive phenotype, the maximum tolerated doses will kill more cancer cells than healthy ones. Without this assumption, it is not clear whether one can expect the same strategy to be optimal, nor whether the patient can efficiently be treated. 

We also emphasise that, for these IDEs, a PMP can be established but would not lead to tractable equations. The key property to still be able to identify the optimal strategy in $\mathcal{B}_T$ is that the long first phase allows us to use Theorem \ref{thm_asympt}: both populations concentrate and their dynamics on the last phase are (approximately) governed by ODEs, as proved in Lemma \ref{Concentrated}. The second phase can thus be analysed with ODE techniques, here the Pontryagin maximum principle (see~\cite{Agrachev2004, Pontryagin1964, Trelat2012}). This is done in Proposition \ref{SecondPhase}.

More concretely, Theorem \ref{SolveOCP} says that:
\begin{quote}
To optimally treat a cancer, the quasi-optimal strategy consists of:
\begin{itemize}
\item
first, administering constant doses to the patient, over a long time. The role of the first long-time arc is to allow the cancer cell population to concentrate on a sensitive phenotype. From a mathematical point of view, this means that the healthy and tumour cell populations (almost) converge to a Dirac mass.
\item
second, during a short-time phase, following a strategy composed of at most three arcs. If the first phase is such that the constraint (\ref{cont_HC}) is saturated, then there can be a first arc along this constraint. The maximal amount of drugs is administered until the constraint (\ref{cont_H}) saturates. The last arc is along this constraint, with an appropriately chosen cytotoxic drug infusion which leads to a further decrease of the number of cancer cells. 
\end{itemize}
\end{quote}

\par 
Numerically, we solve the problem ({\bf OCP}) in $\mathcal{A}_T$. The simulations confirm the theoretical results and show that, with the chosen set of parameters, the strategy indeed approximately consists of these two phases for $T$ large. 
We also compare the optimal strategy with a periodic one, and verify that the former performs better than the latter.
\par
Furthermore, the numerical results suggest that for generic parameters, the optimal choice of constant controls on the first phase is such that the constraint (\ref{cont_HC}) is saturated. Thus, the second phase possibly starts on this constraint.
\par 
Another important property highlighted by the numerical simulations is that, given the choice of parameters made, $\rho_C$ can decrease arbitrarily close to $0$ once the cancer cell population has concentrated on a sensitive enough phenotype. We thus find a strategy for which $T\longmapsto \rho_C(T)$ is decreasing to $0$; hence, there would be no solution to ({\bf OCP}) if the final time $T$ were let free. \par

This is the first time that a mathematical model based on integro-differential equations demonstrates that, within our modeling framework, immediate administration of  maximal tolerated drug doses, or a periodic treatment schedule, is an optimal solution for eradicating cancer. Here, we prove that it is better to allow the phenotypes to concentrate, before administering maximal doses. Such a strategy is much more efficient. This is also a message to be conveyed to clinicians, who have become increasingly aware of such facts.

\bigskip
The paper is organised as follows. Section \ref{Section2} is devoted to the proof of Theorem \ref{thm_asympt} and to numerical simulations showing how the model can reproduce the regrowth of a cancer cell population. Using these results, we have theoretical and numerical grounds for our claim that constant doses are sub-optimal and we then turn our attention to ({\bf OCP}). In Section \ref{Section3}, several arguments are given to justify the restriction to the class $\mathcal{B}_T$, with a long first phase. The rest of the section is then devoted to proving Theorem \ref{SolveOCP}. The numerical solutions of ({\bf OCP}) in $\mathcal{A}_T$ are provided in Section \ref{Section4}. They are compared to periodic strategies. In Section \ref{Section5}, we conclude with several comments and open questions.

\section{Constant infusion strategies}
\label{Section2}

%
This section is devoted to the asymptotic analysis of the IDE model (\ref{contsyst}), in order to specifically understand the effect of giving constant doses on the long run. We start by considering one equation only, for which relatively simple and well known arguments are sufficient to conclude that both convergence and concentration hold. 

\subsection{Asymptotics for healthy or cancer cells alone}
In this section, we assume that $n_H^0=0$ and that $n_C^0$ is continuous and positive on $[0,1]$. We have the following result of convergence and concentration for constant controls and cancer cells alone. Of course, we have a similar statement for healthy cells alone.

\begin{lemma}\label{lem1}
Assume that $u_1$ and $u_2$ are constant: $u_1\equiv\bar u_1$, and $u_2\equiv \bar u_2$, and assume that 
\begin{equation}
\label{Positive}
\frac{r_C}{1+\alpha_C\bar u_2} - \bar u_1 \mu_C > 0 \; \text{on} \; [0,1].
\end{equation}
Then the total population of cancer cells $\rho_C(t)$ converges to $\rho_C^\infty\geq 0$, which is the smallest nonnegative real number such that
\begin{equation}\label{ineq1}
\frac{r_C}{1+\alpha_C\bar u_2} - \bar u_1 \mu_C \leq d_Ca_{CC}\rho_C^\infty \; \, \text{on} \, \; [0,1].
\end{equation}
Let $B_C\subset[0,1]$ be the set of points such that the equality holds in \eqref{ineq1}.
Then $n_C(t,\cdot)$ concentrates on $B_C$ as $t$ goes to $+\infty$. In particular, if $B_C$ is reduced to a singleton $x_C^\infty$, then $n_C(t,\cdot)$ converges to $\rho_C^\infty \delta_{x_C^\infty}$ in $\cal{M}$$^1(0,1)$. 
\end{lemma}
Here and in the sequel, $\delta_x$ denotes the Dirac mass at $x$, and $\mathcal{M}^1(0,1)$ is the set of Radon measures supported in [0,1].

The proof of this lemma is rather classical, its main ingredient is proving that $\rho$ is a $BV$ function to obtain convergence, and concentration follows. It is done in Appendix \ref{AppA}. \par

\begin{remark}
\label{MoreGeneral}
As it clearly appears in the proof, the result holds for more general initial conditions in $L^\infty(0,1)$. One only needs to require that they are bounded from below by a positive constant on a neighbourhood of one of the points of $B_C$. 
\end{remark}


\subsection{Asymptotics for the complete model: proof of Theorem \ref{thm_asympt}}
Now, let us take into account the complete coupling between healthy and tumour cells.
For the remaining part of this article, we assume for simplicity that both $n_H^0$ and $n_C^0$ are continuous and positive on $[0,1]$ (although it is possible to be slightly more general, see Remark \ref{MoreGeneral}), but we emphasise that we no longer require assumption (\ref{Positive}). A further technical assumption is needed to prove that convergence and concentration hold, namely that the functions are Lipschitz continuous:
\begin{equation}
\label{Regularity}
r_H, r_C, d_H, d_C, \mu_H, \mu_C\in{C^{0,1}(0,1)}.
\end{equation}
In two dimensions and with constant controls $\bar u_1$, $\bar u_2$, the previous technique of proving that $\rho$ is $BV$ cannot be extended. As for a single equation, however, we can integrate the equations with respect to $x$ to obtain upper bounds for $\rho_H$ and $\rho_C$. For example, let us integrate the equation defining $\rho_H$ and bound as follows: 
\[
\dfrac{d\rho_H(t)}{dt} \leq \int_0^1 \big(r_H(x) - d_H(x) a_{HH} \rho_H(t)\big) n_H(t,x) \, dx.
\]
Thus, we clearly have $\limsup_{t\rightarrow +\infty}\rho_i \leq \rho_i^{max}:= \max_{x}\frac{r_i}{a_{ii}d_i}$ for $i=H,C$. \par

It also still holds with the reasoning made in the proof of Lemma \ref{lem1} that if $\rho_H$ and $\rho_C$ converge, then the limits must be the solution of the (invertible) system
\begin{equation}
\begin{split}
a_{HH} \rho_H^\infty + a_{HC} \rho_C^\infty &= I_H^\infty , \\
a_{CH} \rho_H^\infty + a_{CC} \rho_C^\infty &= I_C^\infty.
\end{split}
\end{equation}
where $I_H^\infty\geq 0$ is the smallest nonnegative real number such that
\begin{equation}\label{ineq11}
\frac{r_H(x)}{1+\alpha_H \bar u_2} - \bar u_1 \mu_H(x) \leq  d_H(x) I_H^\infty,
\end{equation}
and $I_C^\infty\geq 0$ is the smallest nonnegative real number such that
\begin{equation}\label{ineq22}
\frac{r_C(x)}{1+\alpha_C \bar u_2} - \bar u_1 \mu_C(x) \leq d_C(x) I_C^\infty .
\end{equation}
Furthermore, if this convergence holds true, then $n_H$ (resp. $n_C$) concentrate on $A_H$ (resp. $A_C$) defined as
\begin{align*}
A_H & = \bigg\{x\in[0,1], \; \frac{r_H(x)}{1+\alpha_H \bar u_2} - \bar u_1 \mu_H(x) - d_H(x) I_H^\infty = 0\bigg\}, \\ 
A_C & = \bigg\{x\in[0,1], \; \frac{r_C(x)}{1+\alpha_C \bar u_2} - \bar u_1 \mu_C(x) - d_C(x) I_C^\infty = 0\bigg\}.
\end{align*}

\medskip 
\noindent
\textbf{Proof of Theorem 1.} 

\textit{First step: definition of the Lyapunov functional.} \par
We adapt a strategy developed in~\cite{Jabin2011}. We choose any couple of measures $(n_H^\infty,n_C^\infty)$ in $\cal{M}$$^1(0,1)$ satisfying $\int_0^1 n_{\scriptscriptstyle{i}}^\infty (x) \, dx = \rho_{\scriptscriptstyle{i}}^\infty$, $i=H,C$ which furthermore satisfy 
\begin{equation}
\label{Support}
\text{supp}(n_H^\infty) \subset{A_H}, \; \text{supp}(n_C^\infty) \subset{A_C}.
\end{equation}

For $i=H,C$, and $m_{i} := \frac{1}{d_i}$, let us define the Lyapunov functional as 
\[V(t):= \lambda_H V_H(t) + \lambda_C V_C(t),\]
 where 
\[V_{i} (t) = \int_0^1 m_{i} (x) \left[ n_{i}^\infty(x) \ln\left(\dfrac{1}{n_{i}(t,x)}\right) + \left( n_{i}(t,x) - n_{i}^\infty(x) \right) \right] \, dx, \]
with positive constants $\lambda_H$ and $\lambda_C$ to be adequately chosen later. \par
\medskip
\textit{Second step: computation and sign of the derivative.}  

In what follows, we skip dependence in t in the functions $R_H$ and $R_C$ to increase readability.
We have
\begin{align*}
\dfrac{dV_H}{dt} & = \int_0^1 m_H(x) \left[- n_H^\infty(x) \dfrac{\partial_t n_H(t,x)}{n_H(t,x)} + \partial_t n_H(t,x) \right] \, dx \\ 
			 & = \int_0^1 m_H(x) \, R_H\left( x,\rho_H, \rho_C,u_1,u_2 \right)  \left[n_H(t,x)- n_H^\infty(x) \right] \, dx \\ 
			 & = \int_0^1 m_H(x) \, \left(R_H\left(x,\rho_H, \rho_C,u_1,u_2\right)-R_H\left(x,\rho_H^\infty, \rho_C^\infty,u_1,u_2\right)\right) \left[n_H(t,x)- n_H^\infty(x) \right] \, dx \\
			 &  \hspace{12em} + \int_0^1 m_H(x) \, R_H\left(x,\rho_H^\infty, \rho_C^\infty,u_1,u_2\right)  \left[n_H(t,x)- n_H^\infty(x) \right] \, dx
\end{align*}
The first term is simply 
\begin{align*}
 \int_0^1 m_H(x) \, \left(R_H\left(x,\rho_H, \rho_C,u_1,u_2\right)-R_H\left(x,\rho_H^\infty, \rho_C^\infty,u_1,u_2\right)\right) \left[n_H(t,x)- n_H^\infty(x) \right] & \\
  =\int_0^1 m_H(x) d_H(x) \, \left[a_{HH} (\rho_H^\infty - \rho_H) + a_{HC} (\rho_C^\infty - \rho_C)\right]  \left[n_H(t,x)- n_H^\infty(x) \right]  \, dx \\
 = - a_{HH} (\rho_H^\infty - \rho_H)^2 - a_{HC} (\rho_C^\infty - \rho_C)(\rho_H^\infty - \rho_H)
\end{align*}
The second term can also be written as
\begin{align*}
B_H(t) &:=  \int_0^1 m_H(x) \, R_H\left(x,\rho_H^\infty, \rho_C^\infty,u_1,u_2 \right)  \left[n_H(t,x)- n_H^\infty(x) \right] \, dx \\
&\; =  \int_0^1 m_H(x) \, R_H\left(x,\rho_H^\infty, \rho_C^\infty,\bar u_1,\bar u_2 \right)  \left[n_H(t,x)- n_H^\infty(x) \right] \, dx \\
			 +& \int_0^1 m_H(x) \, \left(R_H\left(x,\rho_H^\infty, \rho_C^\infty,u_1,u_2\right)-R_H\left(x,\rho_H^\infty, \rho_C^\infty,\bar u_1,\bar u_2\right) \right)\left[n_H(t,x)- n_H^\infty(x) \right] \, dx \\
	&\;  = \int_0^1 m_H(x) R_H\left(x,\rho_H^\infty, \rho_C^\infty,\bar u_1,\bar u_2 \right) n_H(t,x)  \, dx \\
	  +  \int_0^1& m_H(x) \left[ r_H(x) \left(\frac{1}{1+\alpha_H u_2}- \frac{1}{1+\alpha_H \bar u_2}  \right) + \mu_H(x) (\bar u_1-u_1)\right] \left[n_H(t,x) -n_H^\infty(x)\right]\, dx, 
\end{align*}
where we use (\ref{Support}) for the last equality. Note that the first term in the last expression is nonpositive by definition of $\left(\rho_H^\infty, \rho_C^\infty\right)$, and the second goes to $0$ as $t$ goes to $+\infty$. Consequently, the decomposition
 \begin{equation}
 \label{Decomp}
 B_{i}= \tilde{B}_{i} + E_{i}, \; \; i=H,C, 
 \end{equation}
holds, with $\tilde{B}_{H}$, $\tilde{B}_{C}$ nonpositive, and $E_H$, $E_C$ which asymptotically vanish. This decomposition will be important in the last step.

Eventually, we have: 
\begin{equation}
\label{Derivative}
\dfrac{dV}{dt} = 
- \dfrac{1}{2} X^T M X + \lambda_H B_H + \lambda_C B_C
\end{equation}
with $M=A^TD+ DA$, 
$X=$
$\begin{pmatrix}
  \rho_H^\infty-\rho_H \\
  \rho_C^\infty-\rho_C \\
 \end{pmatrix}$,
$D=$ 
$\begin{pmatrix}
 \lambda_H & 0  \\
 0 &  \lambda_C  \\
 \end{pmatrix}$ and 
 $A=$
  $\begin{pmatrix}
a_{HH} & a_{HC}  \\
 a_{CH} & a_{CC}  \\
 \end{pmatrix}$
. \\
We first look for a choice of constants $\lambda_H$, $\lambda_C$ that ensures that the symmetric matrix $M$, is also positive semi-definite. \\ 
Since $M=$
 $\begin{pmatrix}
 2  \lambda_H a_{HH} & \lambda_H a_{HC} + \lambda_C a_{CH}  \\
 \lambda_H a_{HC} + \lambda_C a_{CH}   & 2  \lambda_C a_{CC}  \\
 \end{pmatrix}$
has positive trace, both its eigenvalues are non-negative provided that its determinant is non-negative. 
Now, $\det(M) = 4 \lambda_H \lambda_C a_{HH} a_{CC}- \left(\lambda_H a_{HC} + \lambda_C a_{CH}\right)^2$, and we see that choosing $\lambda_H:=\frac{1}{a_{HC}}$ and $\lambda_C:=\frac{1}{a_{CH}}$ leads to 
$$
\det(M) = 4 \, \dfrac{a_{HH}a_{CC}- a_{HC}a_{CH}}{a_{HC}{a_{CH}}}>0
$$
using the assumption (\ref{Competition}). \\
Our aim is to prove that $- \dfrac{1}{2}X^T M X$ converges to $0$ as $t$ goes to $+\infty$, which will yield the convergence of $(\rho_H, \rho_C)$. Concentration of $(n_H, n_C)$ then follows easily with the arguments developed in the proof of Lemma \ref{lem1} . \par
\medskip
\textit{Third step: lower estimate for $V.$}\par
\noindent
To estimate $V$ from below, we need a uniform (in $x$) upper bound on $n_H$, $n_C$. Because of the regularity assumption made on the data (functions are Lipschitz continuous), there exists $C>0$ such that: 
$$ \forall (x,y) \in{[0,1]}, \; \; R_H\left(y, \rho_H, \rho_C,u_1,u_2\right) \geq R_H\left(x, \rho_H, \rho_C,u_1,u_2\right) - C |x-y|,$$
$C$ can be chosen to be independent of $t$ since $\rho_H$, $\rho_C$, $u_1$, and $u_2$ are all bounded. \\
This implies that
\begin{align*}
\int_0^1 n_H(t,y) \, dy &= \int_0^1 n_H^0(y) \exp\left(\int_0^t R_H\left(y, \rho_H, \rho_C,u_1,u_2\right) \, ds \right) \, dy \\
& \geq \int_0^1 \dfrac{n_H^0(y)}{n_H^0(x)} \left(n_H^0(x) \exp\left(\int_0^t R_H\left(x, \rho_H, \rho_C,u_1,u_2\right)\right)\right) \exp \left(- Ct |x-y|\right) \, dy \\
& \geq \dfrac{n_H(t,x)}{n_H^0(x)} \int_0^1 \exp \left(- Ct |x-y|\right) \,dy
\end{align*}
Using the boundedness of $\rho_H$ and $n_H^0$ ($C$ has changed and is independent of $t$ and $x$) and computing the integral, we can write: 
for $t$ large enough, $n_H(t,x) \leq C t $. Similarly, $n_C(t,x) \leq C t.$ \\
The bound on $V$ follows immediately: 
\begin{equation}
\label{BoundV}
V(t) \geq -C \left (\ln(t) + 1\right).
\end{equation}
We now define another function, close to $\frac{dV}{dt}$, whose behaviour will allow us to conclude. \par
\medskip
\textit{Fourth step: estimates on $\frac{dV}{dt}$.} \par
\noindent
We set $$G:= -\dfrac{1}{2} X^T M X +  2 \left(\lambda_H B_H + \lambda_C B_C\right).$$ 
The first term is: 
\begin{align*}-\dfrac{1}{2} X^T M X &= - \lambda_H a_{HH}  (\rho_H^\infty - \rho_H)^2 - \lambda_C a_{CC} (\rho_C^\infty - \rho_C)^2  \\
&  \hspace{10em}- (\lambda_H a_{HC} + \lambda_C a_{CH}) (\rho_H^\infty - \rho_H)(\rho_C^\infty - \rho_C),
\end{align*}
so that (writing in short $R_{H}$ for $R_{H}(x,\rho_H, \rho_C,u_1,u_2)$ and $R_C$ for $R_{C}(x,\rho_C, \rho_H,u_1,u_2)$):
\begin{align*}
\dfrac{d}{dt} \left(-\dfrac{1}{2} X^T M X \right) & = (\rho_H^\infty - \rho_H) \left[ 2 \lambda_H a_{HH} \int_0^1 R_H n_H + (\lambda_H a_{HC} + \lambda_C a_{CH}) \int_0^1 R_C n_C\right] \\
&  \hspace{.4em}+  (\rho_C^\infty - \rho_C) \left[ 2 \lambda_C a_{CC} \int_0^1 R_C n_C + (\lambda_H a_{HC} + \lambda_C a_{CH}) \int_0^1 R_H n_H\right]
\end{align*}
Let us now differentiate $B_H$ and $B_C$ using the expression that initially defines them: 
\begin{align*}
\dfrac{dB_H}{dt} &= \int_0^1 m_H(x) R_H(x,\rho_H^\infty,\rho_C^\infty,u_1,u_2) R_H(x,\rho_H, \rho_C,u_1,u_2) n_H(t,x) \, dx \\
& - \int_0^1 m_H(x) \left(\alpha_H r_H(x) \frac{1}{\left(1+\alpha_H u_2\right)^2} \dfrac{d u_2}{dt} + \mu_H(x)\dfrac{d u_1}{dt} \right) \left[ n_H(t,x) -n_H^\infty(x)\right] \, dx 
\end{align*}
Using the following estimate 
\begin{align*}
 m_H(x) R_H(&x,\rho_H^\infty, \rho_C^\infty,u_1,u_2) \; R_H(x,\rho_H, \rho_C,u_1,u_2) \\ 
&\geq m_H(x) \left(R_H(x,\rho_H^\infty,\rho_C^\infty,u_1,u_2)-R_H(x,\rho_H, \rho_C,u_1,u_2)\right) R_H(x,\rho_H, \rho_C,u_1,u_2) \\
& = - \left[ a_{HH} (\rho_H^\infty- \rho_H) + a_{HC}(\rho_C^\infty- \rho_C) \right] R_H(x,\rho_H, \rho_C,u_1,u_2),
\end{align*}
we obtain the inequality 
\begin{align*}
\dfrac{dB_H}{dt} & \geq - \lambda_H \left[ a_{HH} (\rho_H^\infty- \rho_H) + a_{HC}(\rho_C^\infty- \rho_C) \right] \int_0^1 R_H n_H \\
& - \int_0^1 m_H(x) \left( r_H(x) \frac{1}{\left(1+\alpha_H u_2\right)^2} \dfrac{d u_2}{dt} + \mu_H(x)\dfrac{d u_1}{dt} \right) \left[ n_H(t,x) -n_H^\infty(x)\right] \, dx 
\end{align*}
and similarly for $B_C$. 
We thus obtain (the terms in $a_{HH}$ and $a_{CC}$ cancel each other out):
\begin{align}
\dfrac{dG}{dt} &\geq (\rho_H^\infty - \rho_H) \left[ (\lambda_H a_{HC} + \lambda_C a_{CH}-2 \lambda_C a_{CH} ) \int_0^1 R_C n_C \right]  \nonumber \\
& + (\rho_C^\infty - \rho_C) \left[ (\lambda_H a_{HC} + \lambda_C a_{CH}-2 \lambda_H a_{HC} ) \int_0^1 R_H n_H \right]\nonumber  \\
& - \left(a(t) \dfrac{d u_2}{dt} + b(t)\dfrac{d u_1}{dt}\right)\label{e.dGdt}
\end{align}
where $a$ and $b$ are bounded functions defined by
\begin{align*}
a(t) & :=2 \lambda_H \frac{1}{\left(1+\alpha_H u_2\right)^2} \int_0^1 m_H(x) r_H(x) (n_H(t,x) - n_H^\infty(x)) \, dx \\
  & \hspace{2em} + 2  \lambda_C  \frac{1}{\left(1+\alpha_C u_2\right)^2}\int_0^1  m_C(x) r_C(x) (n_C(t,x) - n_C^\infty(x))\, dx ,
\end{align*}
\begin{align*}
b(t) & :=2 \lambda_H \int_0^1 m_H(x) \mu_H(x) (n_H(t,x) - n_H^\infty(x)) \, dx \\
  & \hspace{2em} +  2 \lambda_C \int_0^1 m_C(x) \mu_C(x) (n_C(t,x) - n_C^\infty(x))\, dx.
\end{align*}
The first two terms at the right-hand side of inequality \eqref{e.dGdt} are equal to $0$ thanks to the choice made for $(\lambda_H, \lambda_C)$ so that it simplifies to 
\begin{equation}
\label{BoundG}
\dfrac{dG}{dt} \geq - \left(a(t) \dfrac{d u_2}{dt} + b(t)\dfrac{d u_1}{dt}\right).
\end{equation}

\textit{Fifth step: conclusion.} \par

\noindent
Noting that $\frac{dV}{dt} \leq \frac{1}{2} G$, it follows that $V(t) - V(0) \leq \frac{1}{2} \int_0^t G(s) \, ds$.
From $G(s)=  G(t) - \int_s^t \frac{d G}{d t}(z) \, dz$, using (\ref{BoundG}) and by integrating the previous inequality, we have
$$\dfrac{V(t) - V(0)}{t} \leq \dfrac{1}{2} G(t) + \dfrac{1}{2t} \int_0^t \int_s^t \left( a(z) \frac{d u_2}{d t}(z) + b(z) \frac{d u_1}{d t}(z)\right) \, dz \, ds. $$  \\
Now, using the decomposition (\ref{Decomp}) introduced in the second step, we obtain: 
\begin{align*}
2 \dfrac{V(t) - V(0)}{t} & - \dfrac{1}{t} \int_0^t \int_s^t \left( a(z) \frac{d u_2}{d t}(z) + b(z) \frac{d u_1}{d t}(z)\right) \, dz \, ds - 2\left(\lambda_H E_H + \lambda_C E_C\right)   \\
& \leq  - \dfrac{1}{2} X^T M X +2 \left( \lambda_H \tilde{B}_H + \lambda_C \tilde{B}_C\right).
\end{align*}
In other words, since the right-hand side of this inequality consists of nonpositive terms, the claim on the convergence of $\rho_H$ and $\rho_C$ is proved if we establish that the left-hand side tends to $0$.\\
As a consequence of the estimate (\ref{BoundV}) on $V$ established in the third step, $2 \dfrac{V(t) - V(0)}{t}$ converges to $0$. This is also true for $2 \left(\lambda_H E_H + \lambda_C E_C\right)$. It thus remains to analyse the asymptotic behaviour of the function $\frac{1}{t} \int_0^t \int_s^t \left( a(z) \frac{d u_2}{d t}(z) + b(z) \frac{d u_1}{d t}(z)\right) \, dz \, ds$.
The analysis relies on the following lemma.

\begin{lemma}
\label{LemmaBV}
Let $\phi$ in $L^\infty(0, +\infty)$, and $u$ in $BV\left([0, +\infty)\right)$. 
Then $$\lim_{t \rightarrow +\infty} \frac{1}{t} \int_0^t \int_s^t \phi(z) u'(z)   \, dz \, ds = 0.$$
\end{lemma}

\begin{proof}
Let us start by writing
\begin{align*}
\dfrac{1}{t} \int_0^t \int_s^t  \phi(z) u'(z)  \, dz \, ds & =  \dfrac{1}{t} \int_0^t \int_0^t \phi(z) u'(z) \, dz \, ds - \dfrac{1}{t} \int_0^t \int_0^s  \phi(z) u'(z)  \, dz \, ds  \\ 
& = \int_0^t \phi(z) u'(z)  \, dz - \dfrac{1}{t} \int_0^t \int_0^s  \phi(z) u'(z) \, dz \, ds \\
& = \Gamma(t) - \dfrac{1}{t} \int_0^t \Gamma(s) \, ds 
\end{align*}
where $\Gamma(t):=\int_0^t \phi(z) u'(z)\, dz$.
The expression above can thus be decomposed as the function $\Gamma$ minus its Ces\`aro average.
To conclude, it suffices that $\Gamma$ has a limit at $+\infty$, which in turn is true as soon as $\phi \, u'$ is integrable on the half-line. This fact is a direct consequence of the boundedness of $\phi$ and the integrability of the derivative of a $BV$ function on $[0, +\infty)$.
\end{proof}
This ends the proof of Theorem \ref{thm_asympt}.

\begin{remark}
The situation differs from~\cite{Jabin2011}  where the non-local logistic term is of the form $\int_0^1 b(x,y) \, n(t,y) \, dy$, with some strong competition assumption on the kernel $b$. In particular, that assumption implies the uniqueness of the ESDs,  which is not necessarily true in our setting.
\end{remark}

\begin{remark}
This theorem means that under general conditions, both populations concentrate and the total number of healthy (resp., cancer) cells converge. In the case of constant controls and when there is selection of a unique phenotype in both populations, it provides a complete understanding of the mapping
\begin{equation}
\label{Mapping}
(\bar u_1,\bar u_2) \longmapsto (x_H^\infty,x_C^\infty, \rho_H^\infty, \rho_C^\infty),
\end{equation}
where $\rho_H^\infty \delta_{x_H^\infty}$ and $\rho_C^\infty \delta_{x_C^\infty}$ are the respective limits of $n_H(t,\cdot)$ and $n_C(t,\cdot)$ in $\cal{M}$$^1(0,1)$, as $t$ goes to $+\infty$.
In particular, if we restrict ourselves to constant controls and a large time $T$, the problem of minimising $\rho_C(T)$ is equivalent to minimising $\rho_C^\infty$ as a function of $(\bar u_1,\bar u_2)$.
\end{remark}

\subsection{Speed of convergence in Theorem \ref{thm_asympt}}

Because it relies on a Lyapunov functional, Theorem \ref{thm_asympt} also yields results on the speeds of convergence when the controls are constant, which we state and analyse separately in the following corollary.
\medskip

\begin{corollary} 
Assume $u_1 \equiv \bar u_1$, $u_2 \equiv \bar u_2$. 
\newline
Then, as $t \rightarrow +\infty$, $\left(\rho_H, \rho_C\right) = \left(\rho_H^\infty, \rho_C^\infty\right) + \mathrm{O}\left(\left(\frac{\ln(t)}{t}\right)^{\frac{1}{2}}\right)$ and concentration occurs at speed $\mathrm{O}\left(\frac{\ln(t)}{t}\right)$, in the following sense: 
\begin{align*}
\int_0^1  m_H(x) R_H\left(x,\rho_H^{\infty}, \rho_C^\infty, \bar u_1, \bar u_2\right) n_H(t,x) \, dx  & =  \mathrm{O}\left(\frac{\ln(t)}{t}\right), \\
\int_0^1 m_C(x) R_C\left(x,\rho_C^{\infty}, \rho_H^\infty, \bar u_1, \bar u_2\right) n_C(t,x) \, dx & =\mathrm{O}\left(\frac{\ln(t)}{t}\right).
\end{align*}
\noindent
In particular, if $A_H$ is reduced to a singleton $x_H^\infty$, then 
\begin{equation*}
\forall \epsilon>0, \; \; \int_{[0,1] \setminus  [x_H^\infty- \epsilon, x_H^\infty + \epsilon]} n_H(t,x) \, dx  = \mathrm{O}\left(\frac{\ln(t)}{t}\right),
\end{equation*}
and similarly for $A_C$.

\end{corollary}

\begin{proof}
The speed of convergence of $\left(\rho_H, \rho_C\right)$ and of the integrals can be obtained by rewriting  the end of the proof, namely that $- \frac{1}{2} X^T M X +2 \left( \lambda_H \tilde{B}_H + \lambda_C \tilde{B}_C\right)$ is bounded from below by $2 \frac{V(t)-V(0)}{t}$, which converges to $0$ as $\mathrm{O}\left(\frac{\ln(t)}{t}\right)$. Since each of those three terms is nonpositive, they all converge to $0$ at the previous speed, and the integrals of interest are nothing but the functions $\tilde{B}_H$ and $\tilde{B}_C$. \par 

For the last statement, we fix $\epsilon>0$ and denote $h:=- m_H R_H\left(\cdot, \rho_H^{\infty}, \rho_C^\infty, \bar u_1, \bar u_2\right) \geq 0$ on $[0,1]$, which by assumption vanishes at $x_H^\infty$ only. We choose $a>0$ small enough such that 
$a \ind{[0,1] \setminus  [x_H^\infty- \epsilon, x_H^\infty + \epsilon]} \leq h$ on $[0,1]$. This enables us to write 
$$ \int_{[0,1] \setminus  [x_H^\infty- \epsilon, x_H^\infty + \epsilon]} n_H(t,x) \, dx \leq \dfrac{1}{a} \int_0^1  m_H(x) R_H\left(x,\rho_H^{\infty}, \rho_C^\infty, \bar u_1, \bar u_2\right) n_H(t,x) \, dx = \mathrm{O}\left(\frac{\ln(t)}{t}\right). $$
The reasoning is the same if $A_C$ is reduced to a singleton.
\end{proof}
\begin{remark}
Although the Lyapunov functional gives us information on the speed of both phenomena in the sense defined above, it does not say whether one of the two is faster. However, if $A_H$ or $A_C$ is reduced to a singleton, the speed $\mathrm{O}\left(\frac{\ln(t)}{t}\right)$ obtained for the convergence to $0$ of the expression $- \frac{1}{2} X^T M X +2 \left( \lambda_H \tilde{B}_H + \lambda_C \tilde{B}_C\right)$ is almost optimal: there cannot exist any $\alpha>1$ such that this sum vanishes like $\mathrm{O}\left(\frac{1}{t^\alpha}\right)$.
This comes from the fact that if it were to hold true, $\frac{d V}{dt}$ would be integrable on the half-line, which would imply the convergence of $V$. This is not possible since either $V_H$ or $V_C$ goes to $-\infty$.
\end{remark}

\subsection{Mathematical simulations of the effect of constant drug doses}

Throughout the study, we will consider the following numerical data, taken from~\cite{Lorz2013a}:
$$r_H(x)=\frac{1.5}{1+x^2},\quad r_C(x)=\frac{3}{1+x^2},$$
$$d_H(x)=\frac{1}{2}(1-0.1x), \quad d_C(x)= \frac{1}{2}(1-0.3x),$$
$$a_{HH} = 1,\quad a_{CC}= 1,\quad a_{HC} = 0.07,\quad a_{CH} = 0.01,$$
$$\alpha_H = 0.01,\quad \alpha_C = 1, $$
$$ u_1^{\max} = 3.5,\quad u_2^{\max} = 7,$$
and the initial data
$$n_H(0,x) = K_{H,0} \exp(-(x-0.5)^2/\varepsilon),\quad n_C(0,x) = K_{C,0} \exp(-(x-0.5)^2/\varepsilon), $$
with $\varepsilon>0$ small (typically, we will take either $\varepsilon=0.1$ or $\varepsilon=0.01$), and where $K_{H,0}>0$ and $K_{C,0}>0$ are such that
$$
\rho_H(0) = 2.7, \; \; \rho_C(0) = 0.5.
$$
The value $\rho_H(0)$ is not the same as in~\cite{Lorz2013a}: it is chosen to be slightly below the equilibrium value of the system with $n_C\equiv 0$, $u_1\equiv0$, $u_2 \equiv 0$, in accordance with the fact that there is \textit{homeostasis} in a healthy tissue. Indeed, we start with a non-negligible tumour which must have (due to competition) slightly lowered the number of healthy cells with comparison to a normal situation. \par
We also define $\rho_{CS}(t) := \int_0^1 (1-x) n_C(t,x)\, dx$,
which may be seen as the total number at time $t$ of tumour cells that are sensitive, and
$\rho_{CR}(t) := \int_0^1 x n_C(t,x)\, dx$,
which may be seen as the total number at time $t$ of tumour cells that are resistant.

Of course, sensitivity/resistance being by construction a non-binary variable, the weights $x$ and $1-x$ are an example of a partition between a relatively sensitive class and a relatively resistant class in the cancer cell population; other choices might be made for these weights, e.g., $x^2$ and $1-x^2$. 

\paragraph{Discussion of the choice for $\mu_H$ and $\mu_C$.}
These functions measure the efficiency of the drugs treatment. The choice done in~\cite{Lorz2013a} is
$$\mu_H(x)= \frac{0.2}{0.7^2+x^2},\quad \mu_C(x)= \frac{0.4}{0.7^2+x^2}.$$
However, with this choice of functions, if we take constant controls $u_1$ and $u_2$, with
$$
u_1(t)=\mathrm{Cst}=u_1^{\max}=3.5,\qquad u_2(t)=\mathrm{Cst}=2,
$$
then we can kill all tumour cells (at least, they decrease exponentially to $0$), and no optimisation is necessary. 
The results of a simulation can be seen on Figure \ref{figancienmu}.

\begin{figure}[H]
\centerline{
\includegraphics[width=10.5cm]{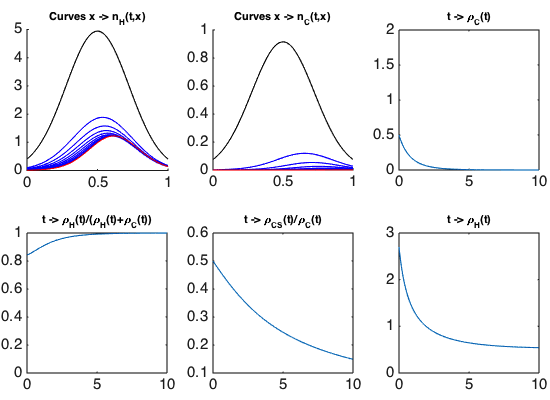}}
\caption{Simulation with $u_1(t)=\textrm{Cst}=3.5$ and $u_2(t)=\textrm{Cst}=2$, in time $T=10$. At the top, left and middle: evolution in time of the curves $x\mapsto n_H(t,x)$ and $x\mapsto n_C(t,x)$, with the initial conditions in black, and the final ones in red. At the right, top and bottom: graphs of $t\mapsto\rho_C(t)$ and of $t\mapsto\rho_H(t)$. At the bottom, left and middle: graphs of $t\mapsto\frac{\rho_H(t)}{\rho_H(t)+\rho_C(t)}$ and of $t\mapsto\frac{\rho_{CS}(t)}{\rho_C(t)}$.}
\label{figancienmu}
\end{figure}

On Figure \ref{figancienmu}, the population of tumour cells is Gaussian-shaped, decreases exponentially to $0$ while its center is being shifted to the right: it means that tumour cells become more and more resistant as time goes by. This is in agreement with the fact that cells acquire resistance to treatment when drugs are given constantly. However, although the proportion of sensitive cells $t\mapsto\frac{\rho_{CS}(t)}{\rho_C(t)}$ is quickly decreasing, the drugs are still efficient at killing the cells. This is not realistic, as it does not match the clinically observed saturation phenomenon. Most cancer cells have acquired resistance and any immediate further treatment should have no effect.

In the simulation above, there is no saturation because the function $\mu_C$ is continuous and positive over the whole interval $[0,1]$ and is not small enough close to $1$.
In order to model this saturation phenomenon, we choose to modify the model used in~\cite{Lorz2013a}, by modifying slightly the function $\mu_C$. The new function $\mu_C$ that will throughout be considered is defined by
$$
\mu_C (x) = \max\left( \frac{0.9}{0.7^2+0.6x^2}-1 , 0\right) .
$$
On Figure \ref{figmuC}, the former function $\mu_C$ is in blue, and the new one is in red.

\begin{figure}[H]
\centerline{\includegraphics[width=6.5cm]{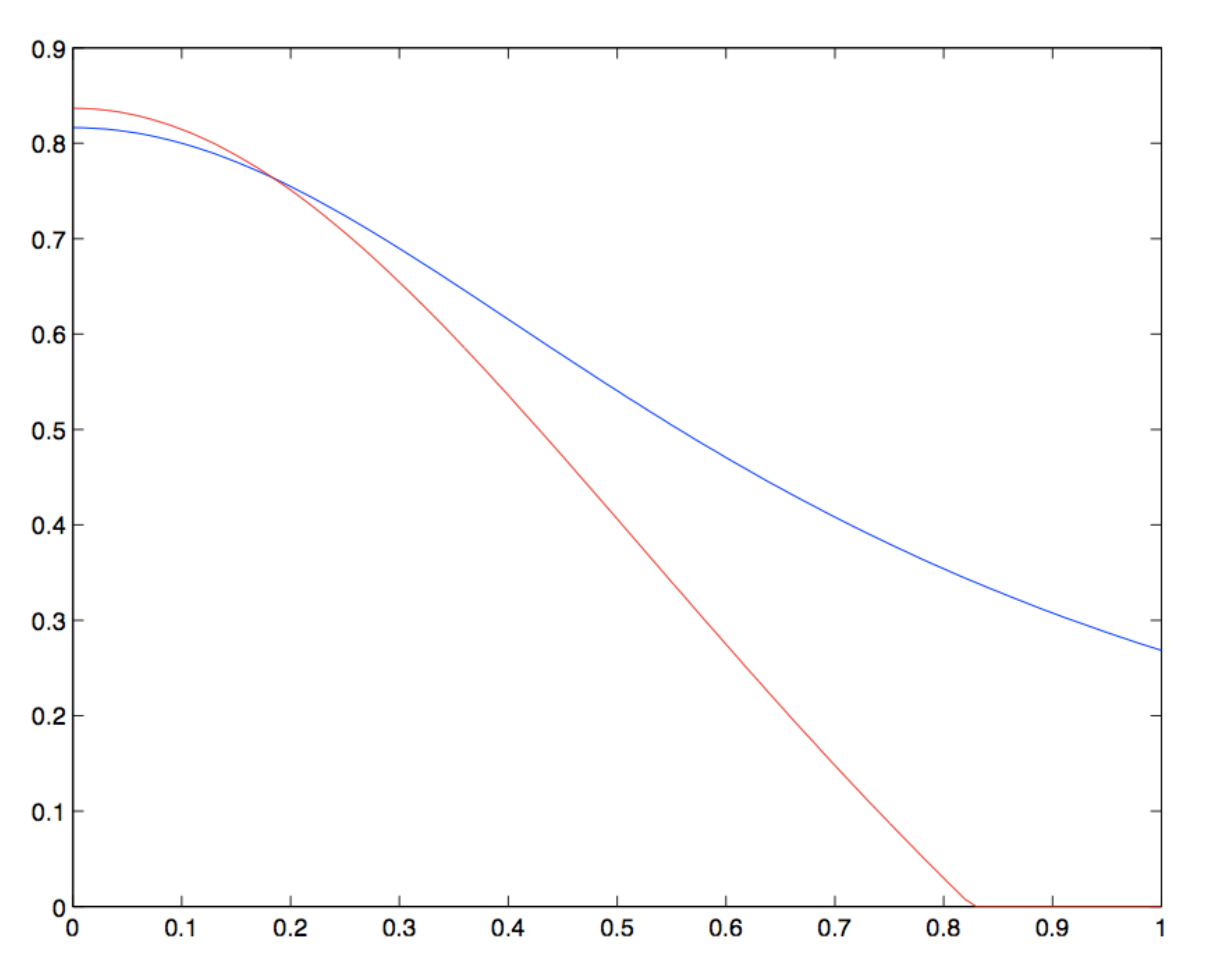}}
\caption{Former function $\mu_C$ in blue, and new function $\mu_C$ in red.}\label{figmuC}
\end{figure}

This new function $\mu_C$ is nonnegative and decreasing on $[0,1]$, and vanishes identically on a subinterval containing $x=1$.

With this new function, the simulation of Figure \ref{figancienmu}, with $u_1=\textrm{Cst}=3.5$ and $u_2=\textrm{Cst}=2$, is completely modified, as can be seen on Figure \ref{fignouveaumu}. Indeed, this time, the strategy consisting of taking constant controls $u_1(t)=\textrm{Cst}=3.5$ and $u_2(t)=\textrm{Cst}=2$ is not efficient anymore and does not allow for (almost) total eradication of the tumour. In sharp contrast, we observe on Figure \ref{figancienmu} that the tumour cells are growing again, moreover concentrating around some resistant phenotype.

\begin{figure}[H]
\centerline{
\includegraphics[width=10.5cm]{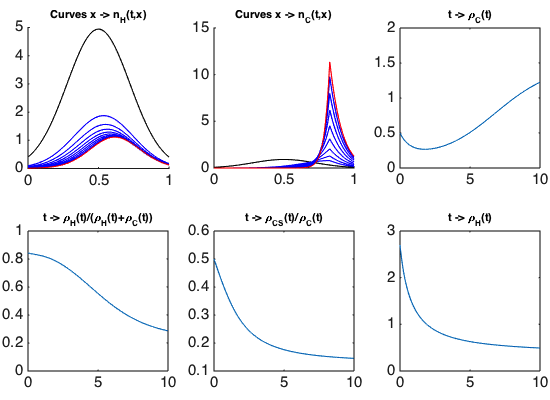}}
\caption{Simulation with $u_1(t)=\textrm{Cst}=3.5$ and $u_2(t)=\textrm{Cst}=2$, in time $T=10$, with the new function $\mu_C$.}\label{fignouveaumu}
\end{figure}
\paragraph{Conclusion on constant controls.} 
The simulations show that choosing  constant doses too high leads to the selection of resistant cells, and then, to regrowth of the cancer cell population if these cells can become  insensitive to the treatment. With the notations of Theorem \ref{thm_asympt}, it is because among constant controls, $(u_1^{max}, u_2^{max})$ does not minimise $\rho_C^\infty$. However, it is quite clear that choosing the optimal constant dose $(\bar u_1,\bar u_2)$ to minimise $\rho_C^\infty$ leaves room for improvement, as the choice $(u_1^{max}, u_2^{max})$ is still the optimal one for sensitive cells. Therefore, it makes sense to allow $u_1$ and $u_2$ to be any functions satisfying \re{contsyst} as in ({\bf OCP}), which we will study both from the theoretical and numerical points of view in the next two sections.

\section{Theoretical analysis of ({\bf OCP})}
\label{Section3}

Before analysing ({\bf OCP}), let us first consider a much simpler ODE model for which we can find the solution explicitly in order to develop some intuition.

\subsection{Simplified optimal control problems}
We consider the ODE 
\beq
\begin{split} 
\dfrac{d\rho}{dt} & =(r-d \rho(t)-\mu u(t)) \rho(t), \la{e.I} \\
\rho(0)& =\rho_0>0.
\end{split}
\eeq

\begin{enumerate}[label=(C\arabic*), start=1]
\item \la{c.1_max} Optimal control problem: minimise $\rho(T)$ over all possible solutions of (\ref{e.I}) with a $L^1$-constraint on $u$, {\it i.e.},
\beq
\int_0^T u(t)\,dt \le u^{1,max}.\la{e.1_max}
\eeq
\end{enumerate}

\begin{lemma}\label{L1}
The optimal solution for problem \ref{c.1_max} is 
$$u_{opt}=u^{1,max} \de_{t=T}.$$
\end{lemma}

\begin{remark}
\label{Relaxed}
The statement can be misleading: we actually prove that there is no optimal solution, but rather that the problem leads to an infimum. Still, writing $u_{opt}=u^{1,max} \de_{t=T}$ makes sense as a way to obtain the infimum is to take a family $(u_\epsilon)_{\epsilon>0}$ in $L^1$ which converges to $u_{opt}$, for example $u_\epsilon:=\frac{1}{\epsilon} u^{1,max} \mathds{1}_{[T-\epsilon,T]}$.
\end{remark}

Adding another the constraint, we have a second optimal control problem
\begin{enumerate}[label=(C\arabic*), start=2]
\item \la{c.1_max_infty_max} minimise $\rho(T)$ over all possible solutions of (\ref{e.I}) with the  $L^1$-constraint \re{e.1_max} and a $L^\infty$-constraint 
\beq
 u \le u^{\infty,max}.\la{e.infty_max}
\eeq
\end{enumerate}
We assume $u^{\infty,max} T > u^{1,max}$, since otherwise it is clear that the optimal strategy is $u^{\infty,max} \ind{[0,T]}$.
\begin{lemma}\label{Linfty}
We define $T_1(T):=T-\frac{u^{1,max}}{u^{\infty,max}}$.
The optimal solution for problem \ref{c.1_max_infty_max}  is 
\beq
u_{op}=u^{\infty,max} \ind{[T_1,T]}. \la{e.u_opt}
\eeq
\end{lemma}
The proofs of these two results can be found in Appendix \ref{AppB}.

\medskip
\begin{remark}
The previous lemmas on simplified equations give some insight on two important features: \par
$\bullet$ for large times, constant controls lead to concentration, as evidenced by Theorem \ref{thm_asympt}. As explained more rigorously further below in Lemma \ref{Concentrated}, when the populations are concentrated on some single phenotypes, the integro-differential equations boil down to ODEs, for which the last results and standard techniques from optimal control theory apply. \par
$\bullet$ for ODE models, it is optimal to use the maximal amount of drug at the end of the time-window if there is a $L^1$ constraint on the control. Avoiding the emergence of resistance will indirectly act as some $L^1$ constraint, which is why this result also provides some interesting intuition on the optimal control problem ({\bf OCP}).
 \end{remark}

\subsection{Assumptions and further remarks}

Let us start by mentioning a possible alternative state constraint for ({\bf OCP}).
\begin{remark}
Alternatively to (\ref{cont_HC}), we might want to directly control the number of cancer cells and replace (\ref{cont_HC}) by 
\begin{equation}\label{cont_rhoC}
\rho_C(t) \leq C^{max}
\end{equation}
for some $C^{max}>0$. The set of constraints (\ref{cont_HC}), (\ref{cont_H}) on the one hand, and (\ref{cont_rhoC}), (\ref{cont_H}) on the other hand, are similar. Although we focus on the first one in the sequel, our analysis applies to the other set of constraints.
\end{remark}

We now make several important additional assumptions which will be used throughout this section, all relying on the notations of Theorem \ref{thm_asympt}. Our first assumption is that
\begin{equation}
\label{OneDirac}
\textit{for any $0\leq \bar u_1 \leq u_1^{max}$, $0\leq \bar u_2 \leq u_2^{max}$, $A_H$ and $A_C$ are reduced to singletons.}
\end{equation}

\noindent
In this case, recall that Theorem \ref{thm_asympt} provides a mapping $(\bar u_1,\bar u_2) \mapsto (x_H^\infty,x_C^\infty, \rho_H^\infty, \rho_C^\infty)$, and with a slight abuse of notation, we will omit the dependence in $(\bar u_1,\bar u_2)$ in the following final assumptions: \par

\begin{quote}
\textit{
whenever $\bar u_1$, $\bar u_2$ are admissible (i.e., such that neither the constraint (\ref{cont_HC}) nor the constraint (\ref{cont_H}) is violated), we require that the solution of the ODE system 
\begin{equation}\label{ODEMTD}
\begin{split}
\frac{d\rho_H}{dt}&= 
R_H\left(x_H^\infty, \rho_H, \rho_C, u_1^{max}, u_2^{max}\right) \rho_H, \\
\frac{d\rho_C}{d t}&= 
R_C\left(x_C^\infty, \rho_C, \rho_H, u_1^{max}, u_2^{max}\right)\rho_C, \\
\end{split}
\end{equation}
with initial data $(\rho_C^{\infty}, \rho_H^{\infty})$, has the following properties:}
\begin{equation}
\label{Decreasing}
\ddt \rho_C<0, \; \ddt \rho_H<0 \; \, \textit{and} \; \, \ddt \f{\rho_C}{ \rho_H}<0.
\end{equation}

\end{quote}
The assumption (\ref{Decreasing}) means that both populations of cells decrease but that the treatment is more efficient on cancer cells. In some sense, this is a curability assumption and it will be crucial in the sequel.
\smallskip

We now motivate the choice of restricting our attention to the class $\mathcal{B}_T$ by giving two results.
\subsection{Optimality of a concentrated initial population for a small time}

Here, we assume that for any $0 \leq \rho_C \leq \rho_C^{max}$, $0 \leq \rho_H \leq \rho_H^{max}$, $0 \leq u_1 \leq u_1^{max}$ and $0 \leq u_2 \leq u_2^{max}$,
\begin{equation}
\label{UniqueMinimum}
\text{$x \mapsto R_C(x, \rho_C, \rho_H, u_1, u_2)$ has a unique minimum.}
\end{equation}

For a given initial amount of cancer cells $\rho_C^0>0$, we define: 
\[ A_{\rho_C^0}:=\left\{ n_C^0\in{\cal{M}}^1(0,1) \; \text{such that} \int_0^1 n_C^0(x) \, dx=\rho_C^0 \right\}. \] 
For $n_C^0 \in{A_{\rho_C^0}}$, and given $n_H^0$ in $\cal{M}$$^1(0,1)$, final time $t_f > 0$, and controls $u_1$, $u_2$ in $BV(0,t_f)$ satisfying (\ref{u_max}), we consider the associated trajectory $(n_H(\cdot,x), n_C(\cdot,x))$ on $[0,t_f]$ solution to the system \re{contsyst} starting from $(n_H^0, n_C^0)$.\par
We consider the following minimisation problem 
\begin{equation}
\label{OCPDirac}
\inf_{\substack{0\leq u_1(t)\leq u_1^{max} \\0 \leq u_2(t) \leq u_2^{max}}} \inf_{n_C^0 \in A_{\rho_{ C^0}}} \rho_C(t_f).
\end{equation}
In other words, for a fixed initial tumour size, we aim at tackling the following question: 
\begin{quote}
\textit{what is the cancer cells' best possible repartition in phenotype}?
\end{quote}
A simpler (and instantaneous) version of the previous optimisation problem for $t_f$ small is 
\begin{equation}
\label{OCPDirac2}
\inf_{\substack{0\leq u_1\leq u_1^{max} \\0 \leq u_2 \leq u_2^{max}}} \inf_{n_C^0 \in A_{\rho_{ C^0}}} \frac{d\rho_C}{dt}(0),
\end{equation}
for which the solution is easily obtained, and given in the following proposition.
\medskip
\begin{proposition}
\label{DiracStart}
Let $g:= R_C\left(\cdot, \rho_C^0, \rho_H^0, u_1^{max}, u_2^{max}\right)$.
We define $x_C$ by $\{x_C\}:=\arg \min g$ and $\tilde{n}_C^0:= \rho_C^0 \delta_{x_C}$. The optimal solution for the optimisation problem (\ref{OCPDirac2}) is given by 
\begin{equation}
\label{SolveDiracStart}
\left(u_1^{max}, u_2^{max}, \rho_C^0 \delta_{x_C}\right).
\end{equation}
\end{proposition}

\begin{proof}
For any $0 \leq u_1 \leq u_1^{max}$, $0 \leq u_2 \leq u_2^{max}$, $n_C^0 \in  A_{\rho_{ C^0}}$, 
\begin{align*}
\dfrac{d \rho_C}{d t} (0)  =\int_0^1 R_C\left(x, \rho_C^0, \rho_H^0, u_1, u_2 \right) n_C^0(x) \, dx \geq  \int_0^1 g(x) \, n_C^0(x) \, dx 
\end{align*}
with equality if and only if $u_1=u_1^{max}$, $u_2=u_2^{max}$. 
\par
We also have $\int_0^1 g(x) \, n_C^0(x) \, dx \geq \int_0^1 g(x_C) \, n_C^0(x) \, dx = g(x_C) \rho_C^0$ and it remains to prove that there is equality if and only if $n_C^0 = \rho_C^0 \delta_{x_C}$. If $n_C^0 \neq \rho_C^0 \delta_{x_C}$there exists $a\in{\text{supp}  \left(n_C^0\right)}$, $a \neq x_C$: it is therefore possible to find $\epsilon>0$ such that both $x_C \not \in [a-\epsilon, a+\epsilon]$ and $\int_{[a-\epsilon, a+\epsilon]} n_C^0(x) \, dx > 0$.
\par 
This implies 
\begin{align*}
\int_0^1 \left( g(x) - g(x_C) \right) n_C^0(x) \, dx \geq 
\int_{[a-\epsilon, a+\epsilon]} \left(g(x) - g(x_C)\right)n_C^0(x) \, dx > 0, 
\end{align*}
which concludes the proof.
\end{proof}

For ({\bf OCP}), the previous Proposition means that, very close to $T$, the best shape of the cancer cell density $n_C(t,\cdot)$ is a Dirac mass. As it was proved in Theorem \ref{thm_asympt}, it is possible (in arbitrarily large time) to reach Dirac masses with constant controls. The combination of these two results is our motivation for the restriction to the set $\mathcal{B}_T$.

\subsection{Reduction of IDEs to ODEs at the end of the long first phase}
Because of the previous result, it makes sense to steer the cancer dell density as close as possible to a Dirac mass. As it was proved in Theorem \ref{thm_asympt}, it is possible (in large time limit) to reach Dirac masses with constant controls. Our aim is now to prove that if we give constant controls $(\bar u_1, \bar u_2)$ for a long time, the dynamics of the total number of cells $(\rho_H, \rho_C)$ are arbitrarily close to being driven by a system of ODEs, a result which comes from the concentration of the IDE on $(x_H^\infty, x_C^\infty)$. The rigorous statement is given hereafter:

\begin{lemma}
\label{Concentrated}
We fix $T_2>0$, $0 \leq \bar u_1\leq u_1^{max}$ and $0\leq  \bar u_2\leq u_2^{max}$. We consider any controls $(u_1,u_2)$ defined on $[0, T_1 + T_2]$ as follows: they are constant equal to $(\bar u_1, \bar u_2)$ on $[0,T_1]$, and any $BV$ functions on $[T_1, T_1 + T_2]$ which satisfy (\ref{u_max}). 
Let $(n_H, n_C)$ be the solution of (\ref{contsyst}) on $[0, T_1 + T_2]$, with corresponding $(\rho_H, \rho_C)$. \par
Then 
\[
\lim_{T_1 \rightarrow +\infty} \sup_{[T_1, T_1 + T_2]} \max\left(|\rho_H - \tilde{\rho}_H|, |\rho_C - \tilde{\rho}_C|\right) = 0,
\] 
where $(\tilde{\rho}_H, \tilde{\rho}_C)$ solves the controlled ODE system
\begin{equation}
\begin{split}
\frac{d\tilde{\rho}_H}{dt}&= 
R_H\left(x_H^\infty, \tilde{\rho}_H, \tilde{\rho}_C, u_1, u_2\right) \tilde{\rho}_H, \\
\frac{d\tilde{\rho}_C}{d t}&= 
R_C\left(x_C^\infty, \tilde{\rho}_C,\tilde{\rho}_H, u_1, u_2\right) \tilde{\rho}_C, \\
\end{split}
\end{equation}
defined on $[T_1, T_1 +T_2]$, starting at $T_1$ from $(\rho_H(T_1), \rho_C(T_1))$.
\end{lemma} 

\begin{proof}
Let $\epsilon>0$.
We focus on the equation on $n_H$ which we integrate in $x$ for any $t \in{[T_1, T_1+T_2]}$: 
\begin{align*}
\frac{d\rho_H}{dt} & = \int_0^1 R_H(x, \rho_H, \rho_C, u_1,u_2) n_H(t,x) \,dx \\
& = R_H\left(x_H^\infty, \rho_H, \rho_C, u_1, u_2\right) \rho_H   \\
&\hspace{3.1cm}+ \int_0^1 \left(R_H(x, \rho_H, \rho_C, u_1,u_2)- 
R_H\left(x_H^\infty, \rho_H, \rho_C, u_1, u_2\right)\right) n_H(t,x) \, dx 
\end{align*}

For the first term, we write 
\begin{align*}
R_H (x_H^\infty, \rho_H,  \rho_C, u_1, u_2)\rho_H &= R_H(x_H^\infty, \rho_H,\rho_C, u_1, u_2) \tilde{\rho}_H  
+ R_H(x_H^\infty, \rho_H, \rho_C, u_1,u_2) (\rho_H- \tilde{\rho}_H) \\
& = \frac{d\tilde{\rho}_H}{dt} + \tilde{\rho}_H d_H(x_H^\infty) \left(-a_{HH} (\rho_H- \tilde{\rho}_H)-a_{HC} (\rho_C- \tilde{\rho}_C) \right) \\
&\hspace{4.5cm}+  R_H\left(x_H^\infty, \rho_H, \rho_C, u_1, u_2\right) (\rho_H - \tilde{\rho}_H)
\end{align*}
This means we end up with
\begin{align*}
\frac{d}{dt}(\rho_H - \tilde{\rho}_H) & = \tilde{\rho}_H d_H(x_H^\infty) \left(-a_{HH} (\rho_H- \tilde{\rho}_H)-a_{HC} (\rho_C- \tilde{\rho}_C)\right)   \\
&+ R_H\left(x_H^\infty, \rho_H, \rho_C, u_1, u_2\right) (\rho_H - \tilde{\rho}_H) \\
&\hspace{1.8cm} + \int_0^1 \left(R_H(x, \rho_H, \rho_C, u_1,u_2)- 
R_H\left(x_H^\infty, \rho_H, \rho_C, u_1, u_2\right)\right) n_H(t,x) \, dx.
\end{align*}

We look at the last term separately: the first two ones are linked to the discrepancy between $\rho$ and $\tilde{\rho}$, while the last one will be small because $n_H$ is concentrated if $T_1$ is large enough. Setting $w:= \max\left(|\rho_H - \tilde{\rho}_H|, |\rho_C - \tilde{\rho}_C|\right)$, we have the differential inequality 
\begin{equation}
\label{Max}
\frac{d}{dt} |\rho_H - \tilde{\rho}_H| \leq C w + \int_0^1 \left(R_H(x, \rho_H, \rho_C, u_1,u_2)- 
R_H\left(x_H^\infty, \rho_H, \rho_C, u_1, u_2\right)\right) n_H(t,x) \, dx 
\end{equation}
for some constant $C>0$.
The last term can be decomposed as 
\begin{align*}
\frac{1}{1+\alpha_H u_2}&  \int_0^1\left(r_H(x) - r_H(x_H^\infty)\right) n_H(t,x) \,dx   \\
&- u_1   \int_0^1\left(\mu_H(x) - \mu_H(x_H^\infty)\right) n_H(t,x) \,dx - I_H  \int_0^1\left(d_H(x) - d_H(x_H^\infty)\right) n_H(t,x) \,dx.
\end{align*}
Note that $u_1$, $\frac{1}{1+\alpha_H u_2}$ and $I_H$ are all bounded on $[T_1, T_1+T_2]$. Thus, if for any generic function $\phi$, $\int_0^1\left(\phi_H(x) - \phi_H(x_H^\infty)\right) n_H(t,x) \,dx$ is arbitrarily small, so is the last quantity. To that end, we write the solution of the IDE in exponential form 
\[
n_H(t,x) = n_H(T_1,x) \exp \left(\int_{T_1}^t R_H\left(x, \rho_H(s), \rho_C(s), u_1(s), u_2(s)\right) \, ds\right),
\]
where the exponential is uniformly bounded on $[0,1]\times [T_1, T_1+T_2]$, which means that $\left| \int_0^1\left(\phi_H(x) - \phi_H(x_H^\infty)\right) n_H(t,x) \,dx \right| \leq C \int_0^1 \left|\phi_H(x) - \phi_H(x_H^\infty) \right| n_H(T_1,x) \,dx $. Since \\$n_H(T_1,\cdot)$ converges to $\rho_H^\infty \delta_{x_H^\infty}$ in $\cal{M}$$^1(0,1)$ as $T_1$ goes to $+ \infty$, this quantity is arbitrarily small. Plugging this estimate into (\ref{Max}) and writing a similar inequality for the equations on the cancer cells, we obtain for $T_1$ large enough $\frac{dw}{dt}  \leq C w + \epsilon$.
We conclude by applying the Gronwall lemma, together with the fact that $w(T_1)=0$. 
\end{proof}

\subsection{Analysis of the second phase}

According to the previous results, for large $T$ and admissible constant controls $(\bar u_1, \bar u_2)$, we arrive at concentrated populations whose dynamics are driven by a system of ODEs. This naturally leads to considering the following optimal control problem, on the resulting ODE concentrated in $\left(x_H^\infty, x_C^\infty\right)$, starting from $\left(\rho_H^\infty, \rho_C^\infty\right)$ at $t=0$. 
For readability, we write $g_H$ for $g_H(x_H^\infty)$ (resp., $g_C$ for $g(x_C^\infty))$ for any function $g_H$ (resp., $g_C$), and we stress that all assumptions made in this subsection are made for all possible admissible constant controls $(\bar u_1, \bar u_2)$. \par 
The ODE system of equations now reads
\begin{equation}
\dfrac{d \rho_H}{dt} = 
\bigg(\underbrace{\frac{r_H}{1+\alpha_H u_2} - d_H I_H- u_1 \mu_H}_{R_H}\bigg) \rho_H, \;
\dfrac{d \rho_C}{dt} = 
\bigg(\underbrace{\frac{r_C}{1+\alpha_C u_2} - d_C I_C- u_1 \mu_C}_{R_C}\bigg) \rho_C.
\end{equation}
For a given $T_2^M>0$, we investigate the optimal problem of minimising $\rho_C(t_f)$ for $t_f \leq T_2^M$ and controls $(u_1, u_2)$ which satisfy (\ref{u_max}), as well as the constraints (\ref{cont_H}) and (\ref{cont_HC}). The constraint (\ref{cont_HC}) rewrites $\frac{\rho_C}{\rho_H} \leq \gamma$ with \[\gamma:= \frac{1-\theta_{HC}}{\theta_{HC}}.\]
Assume that there exists an optimal solution which is the concatenation of free and constrained arcs (either on the constraint (\ref{cont_H}) or (\ref{cont_HC})), with associated times $(t_i)_{1 \leq i \leq M}$. In particular, we thus assume without loss of generality that the parameters are such that 
\begin{equation}
\label{OneOnly}
\textit{both constraints do not saturate simultaneously on an optimal arc.}
\end{equation} 
Then, by the Pontryagin maximum principle for an optimal control problem with state constraints (see~\cite{Vinter2000}), there exists a bounded variation adjoint vector $p=(p_H, p_C)$ defined on $[0,t_f]$, a scalar $p^0 \leq 0$, non-negative functions $\eta_1$ and $\eta_2$ and non-negative scalars $\nu_i$, $i=1,\ldots,M$ such that if we define the Hamiltonian function by
\begin{align*}
H(\rho_H&,\rho_C, p_H, p_C, u_1,u_2) \\
:=& \; p_H R_H \rho_H + p_C R_C \rho_C + \eta_1 (\theta_H \rho_H^0 - \rho_H) + \eta_2 (\rho_C - \gamma \rho_H)\\
= & \; - p_H d_H I_H \rho_H - p_C d_C I_C \rho_C +  \left(\frac{r_H p_H \rho_H}{1+\alpha_H u_2} +  \frac{r_C p_C \rho_C}{1+\alpha_C u_2}\right)    \\
&\hspace{3.2cm} - \left(\mu_H p_H \rho_H + \mu_C p_C \rho_C\right) u_1 + \eta_1 (\theta_H \rho_H^0 - \rho_H) + \eta_2 (\rho_C - \gamma \rho_H),
\end{align*}
we have \par  
\textbf{1.} $p$, $p^0$, $\eta_1$, $\eta_2$ and the $\left(\nu_i\right)_{i=1,\ldots, M}$ are not all zero.  

\textbf{2.} The adjoint vector satisfies 
\begin{equation}
\label{Adjoint}
\begin{split} 
\dfrac{dp_H}{dt} & = - \frac{\partial H}{\partial \rho_H} = - p_H \left(-a_{HH} d_H \rho_H + R_H\right) + a_{CH} d_C p_C \rho_C + \eta_1 + \gamma \eta_2, \\
\dfrac{dp_C}{dt} & =  - \frac{\partial H}{\partial \rho_C} = - p_C \left(-a_{CC} d_C \rho_C +R_C\right) + a_{HC} d_H p_H \rho_H - \eta_2, \\
\end{split}
\end{equation}
with $p_H(t_f) =0$, $p_C(t_f)=p^0$. 

\textbf{3.} $t \longmapsto \eta_1(t)$ (resp. $t \longmapsto \eta_2(t)$) is continuous along (\ref{cont_H}) (resp. (\ref{cont_HC})), and is such that $\eta_1 (\theta_H \rho_H^0 - \rho_H)=0$ (resp. $\eta_2 (\rho_C - \gamma \rho_H)=0$) on $[0,t_f]$. 

\textbf{4.} For any $i=1,\ldots, M$, the Hamiltonian is continuous at $t_i$. If $t_i$ is a junction or contact\footnote{The starting and ending points of a boundary arc are called junction points if they are distinct, and contact points if they coincide (\textit{i.e.}, if the arc is reduced to a singleton).} point with the boundary (\ref{cont_H}) (resp. with the boundary (\ref{cont_HC})), $p_H(t_i^+) = p_H(t_i^-) + \nu_i$, $p_C(t_i^+) = p_C(t_i^-)$ (resp. $p_H(t_i^+) = p_H(t_i^-) + \gamma \nu_i$, $p_C(t_i^+) = p_C(t_i^-) - \nu_i$). 

\textbf{5.} The controls $u_1$, $u_2$ maximise the Hamiltonian almost everywhere. 

We now make several technical assumptions (for all admissible constant controls $(\bar u_1, \bar u_2)$) by requiring
\begin{equation}
\label{Hyp1}
\gamma < \frac{\mu_H}{\mu_C} \frac{\mu_C a_{HH} d_H - \mu_H a_{CH} d_C}{\mu_H a_{CC} d_C - \mu_C a_{HC} d_H}
\end{equation}
(assuming first $\mu_C a_{HH} d_H > \mu_H a_{CH} d_C$, $a_{CC} \mu_H  d_C >a_{HC} \mu_C  d_H$), 
\begin{equation}
\label{NoFullResistance}
\mu_H, \; \mu_C>0,
\end{equation}
\begin{equation}
\label{Hyp3}
 \alpha_H \mu_C r_H< \alpha_C \mu_H r_C, \; \alpha_H \mu_H r_C < \alpha_C \mu_C r_H
\end{equation}
\begin{equation}
\label{Hyp4}
(\alpha _C r_H \mu_C - \alpha_H r_C \mu_H) \left(u_2^{max}\right)^2 
+2 (r_H \mu_C - r_C \mu_H)u_2^{max}+\frac{\alpha_H r_H \mu_C - \alpha_C r_C \mu_H}{\alpha_H \alpha_C} < 0.  
\end{equation}
Note that the two last assumptions are satisfied as soon as $\frac{\alpha_H}{\alpha_C}$ is very small, at least compared to $\frac{\mu_H}{\mu_C}$. This amounts to saying that cytostatic drugs specifically target the cancer cells better than cytotoxic drugs do.

This last necessary condition motivates the definitions
\[
\phi_1:= \mu_H p_H \rho_H + \mu_C p_C \rho_C,
\] 
and (abusively, since this quantity also depends on $t$)
\[
\psi(u_2):= \frac{r_H p_H \rho_H}{1+\alpha_H u_2} +  \frac{r_C p_C \rho_C}{1+\alpha_C u_2}.
\] 

Let us first analyse a constrained arc on (\ref{cont_H}), whenever it is not reduced to a singleton.

\textit{Arc on the constraint (\ref{cont_H})}. \par
First note that $\rho_C =\frac{\rho_C}{\rho_H} \rho_H = \frac{\rho_C}{\rho_H}   \theta_H \rho_H^0$ is bounded from above by $\gamma \, \theta_H \rho_H^0$.
If we differentiate the constraint, we find that $u_1$ and $u_2$ are determined by
\begin{equation}
\label{Feedback1}
\frac{r_H}{1+\alpha_H u_2} - d_H(a_{HH} \theta_H \rho_H^0+ a_{HC} \rho_C) - u_1 \mu_H=0,
\end{equation}
together with the fact that
\begin{align*}
(u_1,u_2)  \in& \amax  \left(\frac{r_H p_H \rho_H}{1+\alpha_H u_2} +  \frac{r_C p_C \rho_C}{1+\alpha_C u_2} - \left(\mu_H p_H \rho_H + \mu_C p_C \rho_C\right) u_1 \right)\\
 = & \amax \frac{p_C \rho_C}{\mu_H} \left(\frac{r_C \mu_H}{1+\alpha_C u_2} -  \frac{r_H \mu_C }{1+\alpha_H u_2}\right) 
\end{align*}
One can check that (\ref{Hyp3}) and (\ref{Hyp4}) are sufficient conditions to have decrease of the function $u_2 \mapsto \frac{r_C \mu_H}{1+\alpha_C u_2} -  \frac{r_H \mu_C }{1+\alpha_H u_2}$. In particular, $u_2=u_2^{max}$ if $p_C<0$, $u_2 = 0$ if $p_C>0$. Thus, the maximisation condition is equivalent to maximising $-\phi_1 u_1$ if $p_C$ does not vanish on the arc. Hence, $\phi_1=0$ when this condition on $p_C$ is fulfilled. We also obtain $u_1$ in feedback form along the arc, and when $p_C$ does not vanish it is given by:
\[ u_1^{b,v} := \frac{1}{\mu_H} \left(\frac{r_H}{1+\alpha_H v} - d_H(a_{HH} \theta_H \rho_H^0 + a_{HC} \rho_C) \right)  \]
where $v=0$ or $v=u_2^{max}$ depending on the sign of $p_C$.
We assume that this is an admissible control, \textit{i.e.}, that it satisfies 
\begin{equation}
\label{Admissible1}
0<u_1^{b,v}<u_1^{max}
\end{equation}
for $v=0$ and $v=u_2^{max}$, and any $0 \leq \rho_C \leq \gamma \theta_H \rho_H^0$.
If $p_C>0$ and $u_2=0$, the dynamics of $\rho_C$ on the arc (\ref{cont_H}) are given by 
\begin{equation}
\label{LogisticODE}
\frac{d \rho_C}{dt} = \frac{1}{\mu_H} \left(r_b - d_b \rho_C\right) \rho_C
\end{equation}
with 
\[r_d:= \left(r_C \mu_H - r_H \mu_C\right)+ \left(a_{HH} d_H \mu_C - \mu_H a_{CH} d_C\right)\theta_H \rho_H^0,  \; d_b:= \left(a_{CC} \mu_H d_C - a_{HC} \mu_C d_H \right), \] which we assume to be positive. 
This autonomous ODE leads to a monotonic behaviour of $\rho_C$. In order to ensure that the boundary control $u_1=u_1^{b,0}$ is not enough to prevent the increase of $\rho_C$ we assume the following 
\begin{equation}
\label{BadMonotony}
\gamma \, \theta_H \rho_H^0< \frac{r_d}{b_d}.
\end{equation}
The previous hypothesis implies that $\rho_C$ will increase on an arc on (\ref{cont_H}) when $p_C>0$. \par
\textit{Arc on the constraint (\ref{cont_HC})}.  \par
If we differentiate the constraint, we find that $R_H = R_C$, \textit{i.e.}, $u_1$ and $u_2$ are related to one another by
\[  \frac{r_H}{1+\alpha_H u_2} - d_H \rho_H (a_{HH} + \gamma a_{HC}) - u_1 \mu_H=\frac{r_C}{1+\alpha_C u_2} - d_C \rho_H (\gamma a_{CC} + a_{CH}) - u_1 \mu_C.\]

We are now set to prove the result:
\medskip
\begin{proposition}
\label{SecondPhase}
Assume (\ref{AssumptionAlpha}), (\ref{OneDirac}), (\ref{Decreasing}), (\ref{OneOnly}), (\ref{Hyp1}), (\ref{NoFullResistance}), (\ref{Hyp3}), (\ref{Hyp4}), (\ref{Admissible1}), (\ref{BadMonotony}) and that there exists an optimal solution which is the concatenation of free and constrained arcs (either on the constraint (\ref{cont_H}) or (\ref{cont_HC})), with associated times $(t_i)_{1 \leq i \leq M}$. \par
Then, the last three possible arcs are: \par
$\bullet$ a boundary arc along the constraint (\ref{cont_HC}). \par
$\bullet$ a free arc with controls $u_1= u_1^{max}$ and $u_2= u_2^{max}$, \par 
$\bullet$ a boundary arc along the constraint (\ref{cont_H}) with $u_2=u_2^{max}$.
\end{proposition}

The proof is technical and can be found in Appendix \ref{AppC}.

\subsection{Solution of ({\bf OCP}) in $\cal{B}$$_T$ for large $T$: proof of Theorem \ref{SolveOCP}}

Recall that we want to solve ({\bf OCP}) for controls $(u_1,u_2) \in{\cal{B}}$$_T$ for large $T$ and small $T_2^{M}$, a choice motivated by the previous results.
For a given $T$, we denote $\left(\bar u_1^{(T)}, \bar u_2^{(T)}\right)$ a choice of optimal values for the constant controls during the first phase.

\begin{theorem}
\label{SolveOCP}
Assume the hypotheses of Proposition \ref{SecondPhase}.
Then asymptotically in $T$ and for $T_2^{M}$ small, there exists at least one solution to ({\bf OCP}) in $\mathcal{B}_T$. More precisely, there exists $\left(\bar u_1^{opt}, \bar u_2^{opt},T_2^{ode}\right)$, $\left(u_1^{ode},u_2^{ode}\right) \in BV\left(0,T_2^{ode}\right)$ such that if we define  the control $(u_1,u_2)$ by
 \begin{equation*}
(u_1,u_2)(t) = \left\{ \begin{array}{lcl}
 \left(\bar u_1^{opt}, \bar u_2^{opt}\right)& \textrm{on} &  \left(0,T-T_2^{ode}\right),\\
\left(u_1^{ode}(t-T+T_2^{ode}),u_2^{ode}(t-T+T_2^{ode})\right)& \textrm{on} & \left(T-T_2^{ode},T\right)
\end{array}\right.
\end{equation*}
then up to a subsequence we have
\[\lim_{T \rightarrow +\infty} \bigg(C_T(u_1,u_2)  - \inf_{(u_1, u_2)\in{\mathcal{B}_T}} C_T(u_1,u_2) \bigg)= 0,\]
meaning that $(u_1,u_2)$ is quasi-optimal if $T$ is large enough.
Furthermore, on $\left(T-T_2^{ode},T\right)$ the optimal trajectory trajectory obtained with $(u_1,u_2)$ is the concatenation of at most three arcs:

$\bullet$ a quasi-boundary arc along the constraint (\ref{cont_HC}), \par
$\bullet$ a free arc with controls $u_1= u_1^{max}$ and $u_2= u_2^{max}$, \par 
$\bullet$ a quasi-boundary arc along the constraint (\ref{cont_H}), with $u_2=u_2^{max}$.
\smallskip
\begin{remark}
By \textit{quasi-boundary arc}, we mean that the quasi-optimal control is such that $(\rho_H, \rho_C)$ almost saturates the constraints, \textit{i.e.}, up to an error vanishing as $T$ goes to $+\infty$.
\end{remark}
\end{theorem}

\begin{proof}
$ $\newline 
Up to a subsequence, still denoted $T$, we can find $\left(\bar u_1^{opt}, \bar u_2^{opt}\right)$ such that $\left(\bar u_1^{(T)}, \bar u_2^{(T)}\right)$ converges to $\left(\bar u_1^{opt}, \bar u_2^{opt}\right)$ as $T \rightarrow+\infty$. These values for the constant controls yield asymptotic phenotypes $\left(x_H^{opt}, x_C^{opt}\right)$ thanks to Theorem \ref{thm_asympt}.
Then, for any choice of time $T_2 \leq T_2^M$ and $BV$ controls $(u_1,u_2)$ on $(T-T_2,T)$,
\begin{equation}
\label{Concen}
\lim_{T \rightarrow +\infty} \sup_{\left[T-T_2,T\right]}\max\left(|\rho_H - \tilde{\rho}_H|, |\rho_C - \tilde{\rho}_C|\right) = 0,
\end{equation} 
with the notations of Lemma \ref{Concentrated}: $\rho$ is obtained from the IDE system, while $\tilde{\rho}$ is obtained from the ODE concentrated on $\left(x_H^{opt}, x_C^{opt}\right)$.
This is a consequence of a slight refinement of Lemma \ref{Concentrated}. Indeed, for $T$ large, the IDE is almost concentrated on some $\left(x_H^{(T)},x_C^{(T)}\right)$ associated to $\left(\bar u_1^{(T)}, \bar u_2^{(T)}\right)$. The formulae for these quantities given by Theorem \ref{thm_asympt} show that $\left(x_H^{(T)},x_C^{(T)}\right)$ converges to $\left(x_H^{opt}, x_C^{opt}\right)$, hence the concentration of the IDE on $\left(x_H^{opt}, x_C^{opt}\right)$ and the result (\ref{Concen}).

As a consequence, the optimal strategy for the ODE, obtained by Proposition \ref{SecondPhase} is also optimal for the IDE, up to an error vanishing as $T$ goes to infinity. We denote $T_2^{ode} \leq T_2^{M}$, $\left(u_1^{ode},u_2^{ode}\right) \in BV\left(0,T_2^{ode}\right)$ the solutions of this optimal control problem. The last statements of the theorem are then a direct consequence of Proposition \ref{SecondPhase} and the assumption that $T_2^M$ is small, since the IDE and ODE trajectories are arbitrarily close. 
\end{proof}

\section{Numerical simulations}
\label{Section4}
In this section, we solve ({\bf OCP}) numerically in the full class $\mathcal{A}_T$. We will compare the results with the previous section, and check that alternative strategies to the one given in Theorem \ref{SolveOCP} are indeed sub-optimal when $T$ is large.
\subsection{Numerical simulations of the solution to ({\bf OCP})}
For a survey on numerical methods in optimal control of ODEs, we refer to~\cite{Trelat2012}. \par 
Here, we use \textit{direct} methods which consist in discretising the whole problem and reducing it to a "standard" constrained optimisation problem.
The IDE structure is dealt with a discretisation in phenotype, which adds to the discretisation in time. The dimension of the resulting optimisation problem becomes larger as the discretisation becomes finer. This method is hence computationally demanding and its numerical implementation requires some care. It relies on combining automatic differentiation and the modelling language \texttt{AMPL} (see~\cite{Fourer2002}) with the expert optimisation routine \texttt{IpOpt} (see~\cite{Waechter2006}). Several different numerical tricks (hot start, numerical refinement, etc) were also needed. \par 

For the simulations, we take $\theta_{HC} = 0.4, \; \theta_H = 0.6, \; \epsilon=0.1.$ We let $T$ take the values $T=30$ and $T=60$. The results are reported on Figures \ref{fig_T30_100_60} and \ref{fig_T60_150_20} respectively. \par 

\begin{figure}[H]
\centering{
\includegraphics[height=9.7cm]{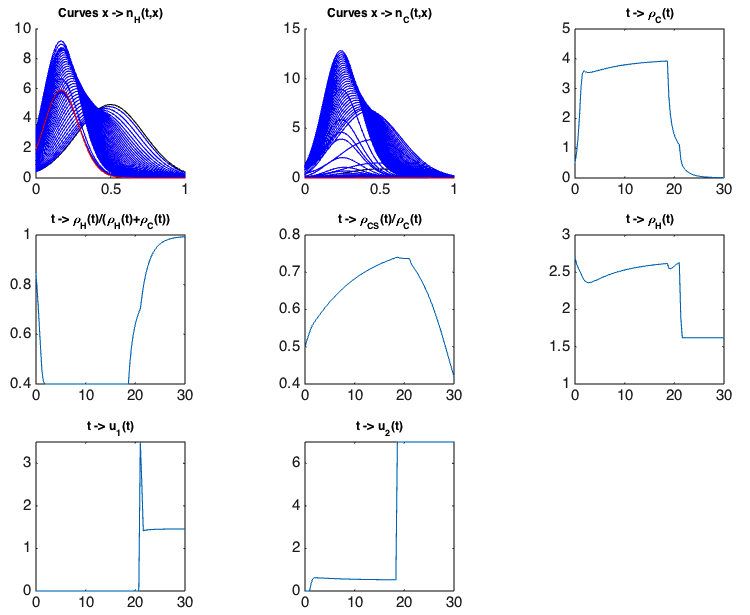}
}
\caption{Simulation of ({\bf OCP}) for $T=30$.} 
\label{fig_T30_100_60}
\end{figure}

\begin{figure}[H]
\centering{
\includegraphics[height=9.7cm]{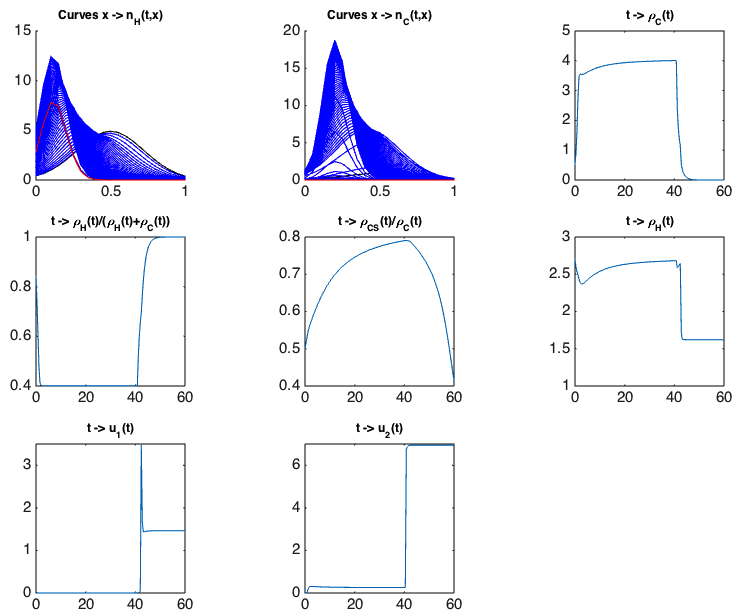}
}
\caption{Simulation of ({\bf OCP}) for $T=60$.}  
\label{fig_T60_150_20}
\end{figure}

These simulations clearly indicate that for the chosen numerical data, if $T$ is large enough, then the optimal controls are such that:
\begin{itemize}
\item the optimal control $u_1$ is first equal to $0$ on a long arc. Then, on a short-time arc, $u_1=u_1^{\max}$ and then to a value such that the constraint (\ref{cont_H}) saturates;
\item the optimal control $u_2$ has a three-part structure, with a long-time starting arc which is a \textit{boundary arc}, that is, an arc along which the state constraint \eqref{cont_HC} is (very quickly) saturated. It corresponds to an almost constant value for the control $u_2$. The last short-time arc coincides with that of $u_1$, and along this arc $u_2=u_2^{max}$.
\end{itemize}

We denote by $t_s(T)$ the switching time, defined by largest time such that 
$u_1(t) = 0$ for all $t<t_s(T)$.

According to the numerical simulations, as $T$ tends to $+\infty$, both $x\mapsto n_C(t_s(T),x)$, $x\mapsto n_H(t_s(T),x)$ converge to (weighted) Dirac masses. Since the controls $u_1$ and $u_2$ are almost constant on $(0,t_s(T))$, this is in accordance with Theorem \ref{thm_asympt}. The cancer cell population is then concentrated on a phenotype on which the drugs are very efficient.

More precisely, as $T$ tends to $+\infty$, the optimal strategy seems tends to a two-piece trajectory, consisting of:
\begin{itemize}
\item a first long-time arc, along the boundary $\frac{\rho_H(t)}{\rho_H(t)+\rho_C(t)}=\delta_{CH}$, with $u_1(t)=0$ and with a constant control $u_2$, at the end of which the populations of healthy and of cancer cells have concentrated on some given sensitive phenotype;
\item a second short-time arc along which the populations of healthy and cancer cells are very quickly decreasing. 
\end{itemize}
We also find that the mapping $T\longmapsto \rho_C(T)$ (where $\rho_C(T)$ is the value obtained by solving ({\bf OCP}) on $[0,T]$) is decreasing. This is because our parameters are such that, once concentrated on a sensitive phenotype, the cancer cell population satisfies a controlled ODE for which there exists a strategy letting $\rho_C$ converge to $0$. Because our model is exponential, we cannot reach $0$ exactly but for very small values of $\rho_C$, one can consider that the tumour has been eradicated.
\medskip

\begin{remark}
In order to avoid additional lengthy hypotheses, we did not give conditions under which the strategy established in Theorem \ref{SolveOCP} can further be identified. However, the numerical solutions show that, for generic parameters, it can be expected that: \par
$\bullet$ the constant controls on the first phase are such that at the end of the first phase, we have saturation of (\ref{cont_HC}), \par
$\bullet$ the second phase is of time duration $T_2^M$ and starts with a constrained arc along (\ref{cont_HC}). \par

\end{remark}

\subsection{Comparison with clinical settings} 
As explained before, our results advocate for a first long phase which must be all the more long for an initially heterogeneous tumour (with respect to resistance). They also apply to 'born to be bad' tumours~\cite{Sottoriva2015}, with high initial heterogeneity with respect to genes or phenotypes in general. Indeed, the heterogeneity or homogeneity we address here is related to one phenotype defined by resistance towards one category of cytotoxic drug. In this sense, our use of the term heterogeneity is unambiguous, functionally defined, and cannot be superimposed on other more classical uses, defined by the accumulation of mutations, such as in~\cite{Ding2012, Gerlinger2012, Sottoriva2015}. 
\par 
This being said, we are ultimately concerned with the application of our optimal control methods to the improvement of classical therapeutic regimens in which repeated courses of chemotherapy are delivered to patients with cancer. To this end, we keep the previous parameters, that are in particular relevant to represent an initially heterogeneous tumour, and we propose for possible implementation in the clinic a quasi-periodic strategy such as in the example defined below:
\begin{itemize}
\item
As long as $\frac{\rho_H}{\rho_H+\rho_C}\geq \theta_{HC}$, we follow the drug-holiday strategy by choosing $u_1=\bar u_1= 0$, $u_2=\bar u_2 = 0.5$ obtained in the previous numerical simulations.
\item Then, as long as $\rho_H > \theta_H \rho_H(0)$, we use the maximal amount of drugs. As soon as $\rho_H = \theta_H \rho_H(0)$, go back to the drug-holiday strategy.
\end{itemize}

The implementation is straightforward, Figure \ref{QuasiPeriodic} shows an example for $T=60$. This strategy allows to maintain the tumour size below some upper value and to prevent resistant cells from taking over the whole population. However, the tumour is not eradicated and this strategy is far from being optimal: $\rho_C(T)$ is slightly below $1$, to be compared to the value obtained with $T=60$ (see Figure\ref{fig_T60_150_20}) with the optimal strategy, which is around $1.10^{-5}$. It is another proof of the importance of a long first phase. It also shows that, at least with our parameters, the last arc on the constraint (\ref{cont_H}) obtained in the previous simulations is instrumental in view of significantly decreasing the tumour size. 
\begin{figure}[H]
\centering{
\includegraphics[height=10cm]{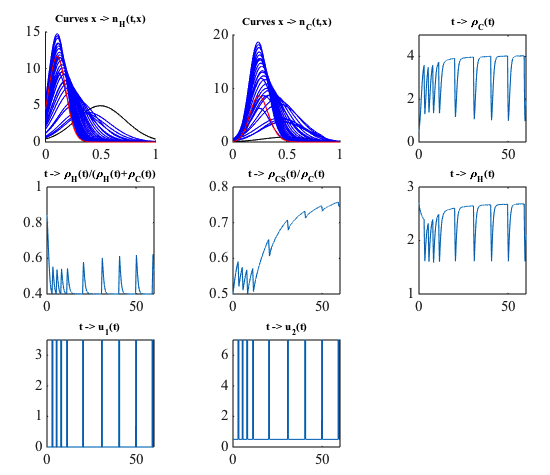}
}
\caption{Quasi-periodic strategy, for $T=60$.}
\label{QuasiPeriodic}
\end{figure}

To assess the importance of the saturation of the constraint $\rho_H = \theta_H \rho_H(0)$, we complement the previous strategy with an arc on this constraint, with $u_2=u_2^{max}$, and adequately chosen feedback control $u_1$ obtained from the equality $\frac{d\rho_H}{dt}=0$. We go back to the drug-holiday strategy as soon as $\rho_C$ starts increasing again, since it is a sign that the tumour has become too resistant. We choose $T=100$ to have enough cycles; the corresponding results are reported on Figure \ref{QuasiPeriodic2} below. They tend to show that $\rho_C$ can be brought arbitrarily close to $0$ after enough cycles, meaning that there is a chance for total eradication of the tumour.
\begin{figure}[H]
\centering{
\includegraphics[height=10cm]{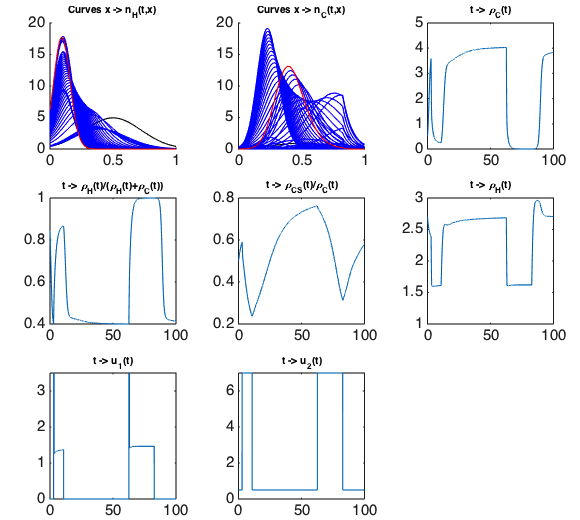}
}
\caption{Second quasi-periodic strategy, for $T=100$.}
\label{QuasiPeriodic2}
\end{figure}

\section{Conclusion}
\label{Section5}

\subsection{Summary of the results} 

By analysing a controlled integro-differential system of cancer and healthy cells structured by a resistance phenotype, we have mathematically investigated the effect of combined chemotherapeutic (cytotoxic and cytostatic) drugs on a tumour. Since we chose a biologically grounded modelling for the resistance phenomenon and took the healthy tissue into account, our approach is tailored for understanding and circumventing the two main pitfalls in cancer therapy: resistance to drugs and toxicity to healthy tissue. The goal of our analysis was indeed twofold: check that our model can reproduce the possible deleterious effect of chemotherapy when MTDs are used (the standard clinical strategy), and propose alternative (optimised) infusion protocols. \par 
Since MTD can first strongly reduce the size of the tumour which then starts growing again, we addressed the first question through an asymptotic analysis of the model. This was performed in Theorem \ref{thm_asympt}, which showed that both populations converge, while the cells concentrate on some phenotypes. This theorem extends results that so far were established only in the case of a single integro-differential equation, although the ideas are radically different because the usual technique (integration w.r.t. the phenotype to show convergence, and then relying on the exponential nature of convergence, to show concentration) does not work in our context. The proof of convergence and concentration, presented in Section \ref{Section2}, relies instead on a suitably defined Lyapunov function, whose analysis gives the speed of convergence and concentration. Interestingly, the approach could incorporate controls which are not only constant, but also asymptotically constant. \par
The rest of the article was then devoted to addressing the second question, by considering the optimal control problem ({\bf OCP}) of minimising the number of cancer cells on a given time interval $[0,T]$, keeping the tumour size in check and limiting damage to the healthy tissue. In Section \ref{Section3}, we gave several rigorous mathematical arguments to explain why, when $T$ is large, a good strategy is to first steer the cancer cell population on an appropriate phenotype by first giving constant doses for a long time. These arguments justified a restriction to a smaller class of controls for which we managed to identify an asymptotically optimal strategy in large time, presented in Theorem \ref{SolveOCP}. \par
In Section \ref{Section4}, we showed through numerical simulations that, when $T$ increases, the optimal solution is indeed increasingly close to a two-phase trajectory. The first very long phase consists in giving low doses of drugs in order to let the cancer cell population concentrate on a given sensitive phenotype. The doses are chosen as low as the constraint on the relative tumour size allows it. Our results advocate for a first long phase which must be all the more long for an initially heterogeneous tumour (with respect to resistance). During the second phase, we numerically recover the expected trajectory, given by Theorem \ref{SolveOCP}: high doses are given (MTD as long as the constraint on the healthy tissue does not saturate) and the cancer cell population quickly decreases. \par

\subsection{Possible generalisations} 
We have focused on a $1$ dimensional phenotype $x \in [0,1]$. In applications, however, it might be suitable to consider multi-dimensional phenotypes in order to account for the complexity of resistance. This is for example what is done in~\cite{Chisholm2015} where the relevant phenotype is $2$ dimensional and decided to be the combination of proliferation and survival potentials. With some technical adaptations, the results established in this paper generalise to any compact subset of $\mathbb{R}^d$, $d\geq1$. For the application one would need to specify how the functions depend (monotonically or not) on the various components of the phenotype. \par
A possible generalisation for our model is to take into account the fact that cells can change phenotype, for instance through (random) genetic \textit{mutations}, \textit{i.e.}, modifications of the DNA. These are irreversible and are passed from one cell to its daughter cells through division. However, it is now widely believed that such mutations are very rare with respect to the typical timescales that are of medical interest (which in the case of drug resistance phenomena are much shorter than the timescale of a human life, a time during which mutations certainly exist and can explain the development of diseases, see, e.g.~\cite{hirsch2016} about acute myeloid leukaemia), and thus they can be disregarded at least as a first approximation. In contrast, \textit{epimutations} (\textit{i.e.}, mechanisms which alter gene expression but not the DNA sequence base pairs themselves) are thought to be much more frequent~\cite{brown2014, shen2013, you2012}. \par
If epimutations are exclusively random, they can be modelled by a diffusion term, in which case (\ref{contsyst}) would be modified as follows:
\begin{equation}\label{contsyst2}
\begin{split}
\dfrac{\partial n_H}{\partial t} (t,x) &= R_H\left(x, \rho_H(t), \rho_C(t),u_1(t),u_2(t)\right) n_H(t,x) + \beta_H \dfrac{\partial^2 n_H}{\partial x^2} (t,x),
\\ \vspace{.8em}
\dfrac{\partial n_C}{\partial t} (t,x)& = 
R_C\left(x, \rho_C(t), \rho_H(t), u_1(t), u_2(t)\right) n_C(t,x) + \beta_C \dfrac{\partial^2 n_C}{\partial x^2} (t,x), 
\end{split}
\end{equation}
together with Neumann boudary conditions in $x=0$ and $x=1$. Here, $\beta_H$ and $\beta_C$ stand for the random epimutations rates of the healthy and cancer cell populations, respectively. \par


The Darwinian idea that the fittest individuals are selected exclusively because of random events affecting the genome or its expression has been recently challenged as observations on genomic evolution cannot be accounted for by sheer Darwinian mechanisms~\cite{Ling2015}, maybe also as ideas from Lamarck are regaining popularity. Such theories advocate the existence of adaptive behaviours: individuals actively adapt to their environment, seeking for phenotype changes that make them fitter. These can be seen as \textit{stress-induced epimutations} and can be mathematically modelled by an advection term, as in~\cite{Chisholm2015, Lorenzi2015}. This would lead to a model of the form 
\begin{equation}\label{contsyst3}
\begin{split}
\dfrac{\partial n_H}{\partial t} (t,x)  + \dfrac{\partial }{\partial x} (v_H\left(x,u_1(t),u_2(t))\, n_H(t,x)\right) &  \\
 = R_H\, (x,\rho_H(t)&, \rho_C(t),u_1(t),u_2(t)) \,n_H(t,x) + \beta_H \dfrac{\partial^2 n_H}{\partial x^2} (t,x), \\
\\ 
\dfrac{\partial n_C}{\partial t} (t,x)  + \dfrac{\partial }{\partial x} (v_C\left(x,u_1(t),u_2(t))\, n_C(t,x)\right) &  \\
 = R_C\, (x,\rho_C(t)&, \rho_H(t),u_1(t),u_2(t)) \,n_C(t,x) + \beta_C \dfrac{\partial^2 n_C}{\partial x^2} (t,x), 
\end{split}
\end{equation}
together with Neumann boudary conditions in $x=0$ and $x=1$. Here, $v_H$ and $v_C$ are the velocities with which healthy and cancer cells adapt to their environment, respectively, which are assumed to vanish in $x=0$ and $x=1$. Because we assume that the adaptation of cells is induced by the stress created by the drugs, $v_H$ and $v_C$ depend on $u_1$ and $u_2$. \par
Another extension could involve a mixed deterministic/stochastic framework, namely using a piecewise deterministic Markov process (PDMP~\cite{davis1984}, see~\cite{Renault2016} for the optimal control of this class of equations). In these models, mutations are stochastic jumps between deterministic (and phenotypically reversible) models, each jump becoming less and less rare in the course of phenotypic evolution in the deterministic processes. Furthermore, in these models, the probability of jump would depend exclusively on (and as an increasing function of) the phenotype structure variable, that would thus bear a quantitative meaning of malignancy, or phenotype plasticity entraining genetic instability (this last point is discussed with references in~\cite{Chisholm2016}). \par
A final extension should stem from the fact that tumours are also very heterogeneous in space (for example, because cells at the outer rim and cells at the centre of a tumour spheroid encompass very different metabolic conditions; more genenerally, high heterogeneity depending on space has been experimentally shown in solid tumours~\cite{Gerlinger2012, Sottoriva2015}, which should lead to also structure the populations of cells according to an added space variable. Another modelling advantage of such representation is that the interaction between the tumour and the healthy tissue is itself spatial, since part of it essentially happens at the boundary of the tumour, through direct contact. For possible cancer models taking both phenotype and space into account, we refer to~\cite{Jabin2016, Lorz2015, Mirrahimi2015}.

\subsection{Open problems}

{\bf Asymptotic analysis.} 
In this paper, we have extended well-known results for a single IDE to systems of IDEs. However, in applications it could be interesting to consider the case of general competitive of the form
\begin{equation}
\label{2x2Extended}
     \dfrac{\partial n_H}{\partial t} (t,x) = R_H \left(x, \rho_H, \rho_C\right) n_H(t,x),   \;
     \dfrac{\partial n_C}{\partial t} (t,x) = R_C \left(x, \rho_C, \rho_H\right) n_C(t,x),
\end{equation}
with $R_H$ decreasing in $\rho_C$, $R_C$ decreasing with $\rho_H$. Proving convergence and concentration for such systems is completely open. Indeed, our Lyapunov function is specifically suited to the specific linear setting of the model (\ref{contsyst}): it cannot be applied to any general competitive system. 
Note that some numerical simulations indicate that no oscillations occur, which may mean that $\rho_H$ and $\rho_C$ converge. \par
Similarly, characterising the asymptotic behaviour of PDE systems like (\ref{contsyst2}) and (\ref{contsyst3}) (even without the control terms) is an interesting and open problem, even for a single equation. Let us mention that when the rate of mutations is small, much can be found in the literature on the asymptotics of these models when this small parameter goes to $0$, after a proper rescaling of time~\cite{Barles2009, Lorz2011}. If this parameter is fixed, classical asymptotics for $t$ going to infinity have up to our knowledge not been carried out. 
\medskip

\noindent
{\bf Optimal control.}
In order to try and solve the optimal control problem ({\bf OCP}), we had to treat the question of the optimal control of IDEs. Although a PMP exists for such equations, we took another path because the resulting equations were too intricate. \par
A key idea to justify the restriction to a class of controls which are first constant on a long phase is to prove that the tumour (among all possible tumours of given size) which can be treated the most efficiently is homogeneous in phenotype, \textit{i.e.}, a Dirac mass in mathematical terms: the first long phase then aims at approaching this 'ideal' situation to start the second phase. We established the optimality of Dirac masses for a short time, but not for any time.

\par

If we want to analyse similar optimal control problems for (\ref{contsyst2}) or (\ref{contsyst3}), we shall have to deal with optimal control of PDEs for which techniques are very different~\cite{Coron2007,Tucsnak2009}. A first approach would be to focus on (\ref{contsyst2}), and see whether and how the optimal controls converge to the ones obtained in this paper as the rates of epimutations go to $0$.

\appendix 
\section{Proof of Lemma \ref{lem1}} 
\label{AppA}
\begin{proof}
We are going to prove that $\rho_C$ is a $BV$ function. To that end, let us prove that $\rho_C$ is bounded from above, and that it has integrable negative part. \par
\medskip
\textit{First step: upper bound for $\rho_C$.} \par
\noindent
The existence of such a bound comes from integrating the equation with respect to $x$: 
$$
\rho_C'(t) = \int_0^1 R_C\left(x,\rho_C, 0 \bar{u}_1, \bar{u}_2 \right) n_C(t,x) \, dx.
$$
If $\rho_C$ is too large, the right hand side is negative, forcing $\rho_C$ to decrease. It proves the claim on the upper bound for $\rho_C$. \par
Similarly, because of assumption (\ref{Positive}), $\rho_C$ increases if $\rho_C$ is too close to $0$: $\rho_C$ is bounded from below by some $\rho_C^{min}>0$. \par
\medskip
\textit{Second step: estimate on the negative part of $\rho_C$.} \par
\noindent
We define $q_C:=\dfrac{d\rho_C}{dt}$ and wish to prove that $(q_C)_{-}\in{L^1(0,+\infty)}$.
We differentiate $\dfrac{d\rho_C}{dt}=\int_0^1 n_C R_C$ to obtain:
\begin{equation*}
\dfrac{dq_C}{dt}=\int_0^1 n_C R_C^2 +\left( \int_0^1  n_C \frac{\partial R_C}{\partial \rho_C} \right) q_C
\end{equation*}
It provides an upper bound for the negative part of $q_C$:
\begin{align*}
\dfrac{d(q_C)_{-}}{dt}   \leq \left( \int_0^1 n_C \frac{\partial R_C}{\partial \rho_C} \right) (q_C)_{-}  \leq -a_{CC} d_C^{min} \rho_C^{min} (q_C)_{-} 
\end{align*}
where $0<d_C^{min}\leq d_C$ on $[0,1]$. 
We conclude that the negative part of $q_C$ vanishes exponentially (and consequently, is integrable over the half-line). Therefore, $\rho_C$ converges to some $\rho_C^\infty>0$.\par
\medskip
\textit{Third step: identification of $\rho_C^\infty$.} \par
\noindent
Now, we have
$$
n_C(t,x) = n_C^0(x)\exp \left( \left( \frac{r_C(x)}{1+\alpha_C\bar u_2} - \bar u_1 \mu_C(x) \right) t - d_C(x)a_{CC}\int_0^t\rho_C(s)\, ds \right)
$$
For large $t$, we have $\int_0^t\rho_C(s)\, ds\sim \rho_C^\infty t$ and hence the asymptotic behaviour depends on the function $b_C$ defined on $[0,1]$ by
$$
b_C(x) = \frac{r_C(x)}{1+\alpha_C\bar u_2} - \bar u_1 \mu_C(x) - d_C(x)a_{CC}\rho_C^\infty.
$$ 
Let $B_C\subset [0,1]$ be the set of points at which the function $b_C$ reaches its maximum. 

Let us prove that $b_C(x)=0$, for every $x\in B_C$.
We argue by contradiction. If $b_C(x)>0$ for some $x\in B_C$, then there exists a nontrivial interval $I\subset[0,1]$ along which $b_C$ is positive, and therefore $n_C(t,\cdot)$ takes larger and larger values along $I$ as $t$ increases. This contradicts the fact that $\rho_C(t)$ converges to $\rho_C^\infty$.
Similarly, if $b_C<0$ globally, $\rho_C$ converges to $0$, a contradiction. 

The function $b_C$ is thus nonpositive on $[0,1]$, and vanishes at any point of $B_C$.
The lemma follows easily.
\end{proof}

\section{Proofs for the simplified optimal control problems} 
\label{AppB}

\subsection{Proof of Lemma \ref{L1}}

\begin{proof}
Using the family $u_\epsilon$ defined in Remark \ref{Relaxed}, we obtain the corresponding $\rho_\epsilon(T)$, which can be computed exactly, as well as its limit. It is given by $$\rho_{opt}(T):= \rho_{opt}(T^{-}) \, \text{exp}(-\mu u^{1,max})$$ where $\rho_{opt}$ is the function obtained through $\ddt \rho_{opt}(t)=(r-d \rho_{opt}(t)) \rho_{opt}(t)$ for $t<T$, and $\rho_{opt}(0)=\rho_0$.\\
Now, let any $u$ satisfy \re{e.1_max}. The solution of (\ref{e.I}) with $u$ is thus a subsolution of that satisfied by $\rho_{opt}$, leading to $\rho \leq \rho_{opt}$ on $[0,T)$. Using $u\geq 0$, we also have $$\rho(T) \geq \rho_0 \, \text{exp} \left(\int_0^T (r - d \rho(s)) \, ds \right) \,  \text{exp}(-\mu u^{1,max}). $$
Since $\rho_{opt} (T^{-}) = \rho_0 \, \text{exp} \left(\int_0^T (r - d \rho_{opt}(s)) \, ds \right)$ and $\rho \leq \rho_{opt}$, this implies $\rho(T) \geq \rho_{opt}(T)$. \\
Let us now investigate the possible case of equality to prove that the infimum is not attained: the foregoing equality implies that there is equality if and only if $\int_0^T u\,ds = u^{1,max}$ (the contraint is saturated) and $\text{exp} \left(\int_0^T (r - d \rho(s)) \, ds\right)=\text{exp} \left(\int_0^T (r - d \rho_{opt}(s)) \, ds\right)$, whence $\rho \equiv \rho_{opt}$ on $[0,T)$. As $\rho$ is continuous, $\rho(T)$ would be given by taking $u\equiv 0$, which is not optimal.
\end{proof}

\subsection{Proof of Lemma \ref{Linfty}}

\begin{proof}
To account for the $L^1$ constraint \ref{c.1_max}, we augment the system by defining another state variable $y$, whose dynamics are given by $\frac{dy}{dt} =u$, leading to:
\begin{equation}
\label{AugmentedSystem}
\begin{split} 
\dfrac{d\rho}{dt} &  =(r-d \rho-\mu u) \rho, \; 
\dfrac{dy}{dt}  =u,\\
\rho(0) & =\rho_0,  \; \; \;\; \; \;\; \; \;\; \; \;\; \; \;\; \; \; \; 
y(0) =0.
\end{split}
\end{equation}
The constraint \ref{c.1_max} thus rewrites $y(T) \leq u^{1,max}$. \par

\noindent 
According to the Pontryagin maximum principle (see~\cite{Pontryagin1964}), there exist absolutely continuous adjoint variables $p_\rho$ and $p_y$ on $[0,T]$, and $p^0\leq 0$, such that: 
\begin{equation}
\label{Costate}
\dfrac{dp_\rho}{dt} = - \frac{\partial H}{\partial \rho} = - (r- 2 d \rho-\mu u) \, p_\rho, \;  
\dfrac{dp_y}{dt} = - \frac{\partial H}{\partial y} =0 
\end{equation}
where the Hamiltonian is $$H(\rho,y,p_\rho, p_y, u):= p_\rho \left(r-d\rho - \mu u\right) \rho + p_y u = \left(r - d\rho\right)p_\rho + u \left(p_y - \mu p_\rho \rho \right).$$  
Thus, $p_y$ is some constant, and $p_\rho$ does not change sign on $[0,T]$. \\
The maximisation of the Hamiltonian leads to defining the switching function $\phi:= p_y - \mu p_\rho \rho$. $u$ is thus equal to $u^{\infty,max}$ whenever $\phi>0$, equal to $0$ whenever $\phi<0$. \par
The transversality condition is that the vector
 $\begin{pmatrix}
  p_\rho \\
  p_y \\
 \end{pmatrix}(T)$
 $- p_0$ 
$\begin{pmatrix}
 1 \\
 0 \\
 \end{pmatrix}$ 
must be orthogonal to the tangent space of $\{ (p,y) \in{\mathbb{R}^2}, y \leq u_{1,max} \}$ at the point $\left( \rho(T), y(T) \right)$. \par

\textit{First case.} \par
If $y(T) < u^{1,max}$, then the transversality conditions imply $p_\rho(T)=p_0$ and $p_y \equiv 0$. $p^0\neq 0$ since otherwise we would have $(p_\rho, p_y, p^0) \equiv 0$. Thus, in this case, $p_\rho(T) < 0$ and $p_\rho$ is negative on the interval $[0,T]$. The switching function $\phi$ is therefore positive on the whole $[0,T]$, which would imply $u\equiv u^{\infty, max}$. This is a contradiction since a consequence is $\int_0^T u(s) \, ds = u^{\infty, max} T > u^{1,max}$.  \par

\textit{Second case.} \par
If $y(T) = u^{1,max}$, we still have $p_\rho(T)=p_0$. As in the first case, we cannot have $p_y=0$.\\
\noindent
Let us first remark that $\phi$ cannot be positive nor negative on the whole interval, since otherwise $u\equiv u^{\infty,max}$, a contradiction, or $u\equiv 0$, which is clearly not optimal. 
If $p_0=0$, $p_\rho \equiv 0$, so that $\phi$ has the sign of $p_y\neq 0$, a contradiction. Therefore, $p_\rho < 0$ on $[0,T]$ as before, and this implies $p_y<0$ to ensure that $\phi$ changes sign. 

\noindent
The derivative of $\phi$ is given by $\frac{d \phi}{d t} = -\mu d p_\rho \rho^2 >0$. Thus, $\phi$ is increasing and $u$ is bang-bang with one switching only. The fact that $y(T)=\int_0^Tu(s) \, ds= u^{1,max}$ imposes that this switching happens at $T_1(T)$ as announced, which ends the proof.
\end{proof}

\section{Proof of Proposition \ref{SecondPhase}}
\label{AppC}

\begin{proof}
If the constraint (\ref{cont_HC}) does not saturate on the whole $[0,t_f]$, we distinguish on whether the last arc is a free arc or a boundary arc on (\ref{cont_H}). 

\textit{First case: the last arc is a boundary arc on (\ref{cont_H}), not reduced to a singleton.}

In this case, $t_M = t_f$ and there can be a jump on the adjoint vector at $t_f$.

Let us start by proving the following: 
\begin{lemma}
$p^0<0$.
\end{lemma}
\begin{proof}
We argue by contradiction and assume $p^0=0$. 
We first look at the interval $[t_{M-1},t_f]$, and assume, also by contradiction, that $\nu_M >0$.
Then $p_H(t_f^-) = -\nu_M <0$, hence $p'_C(t_f^-) <0$, leading to $p_C>0$ in a right neighbourhood of $t_f$. From assumption \eqref{BadMonotony}, this means that $\rho_C$ decreases locally around $t_f$, a contradiction since $t_f$ is free (a better strategy would be to stop before $\rho_C$ starts increasing): $\nu_M = 0$. \par
Now, let us prove that $p_H$, $p_C$ and $\eta_1$ vanish identically on $[t_{M-1},t_f]$. If we have $p_C(t_0)>0$ (resp., $p_C(t_0)<0$) for some $t_0 \in [t_{M-1},t_f)$, we define the maximal interval $[t_0, t^\star)$ on which $p_C>0$ (resp., $p_C<0$), with $p_C(t^\star)=0$. In this case, we know that the switching function $\phi_1$ vanishes on $[t_0, t^\star]$, hence $p_H$ factorises with $p_C$. Coming back to the equation on $p_C$, we have $p'_C = \beta_C p_C$  on $(t_0,t^\star)$, for some function $\beta_C$. Since $p_C(t^\star) =0$, this imposes $p_C\equiv 0$ on the interval, a contradiction. Thus $p_C$ is identically $0$ on the whole $(t_{M-1}, t_f)$, and so are $p_H$ (from the equation on $p_C$) and $\eta_1$ (from the equation on $p_H$). \par
We now analyse the arc $[t_{M-2},t_{M-1}]$. From the previous step, we know that $\phi_1(t_{M-1})=0$. If $\nu_{M-1}>0$, then $\phi_1(t_{M-1}^{-})< \phi_1(t_{M-1}) = 0$, thus $u_1=u_1^{max}$ locally on the left of $t_{M-1}$. Similarly, maximising $\psi(u_2)$ imposes $u_2=u_2^{max}$.
Also, $H\left(t_{M-1}\right)=0$, and $H\left(t_{M-1}^{-} \right)= - \nu_{M-1} R_H\left(t_{M-1}^-\right) \rho_H(t_{M-1})$. By continuity of the Hamiltonian, we get $R_H\left(t_{M-1}^-\right)=0$. At the left of $t_{M-1}$, $u_1$ and $u_2$ saturate at their maximal values. At the right of $t_{M-1}$, $R_H=0$ but this imposes $u_1<u_1^{max}$ or $u_2<u_2^{max}$ since, owing to \eqref{Decreasing}, $\rho_H$ decreases for the maximal values. Thus, $0=R_H\left(t_{M-1}^-\right) < R_H(t_{M-1})=0$, a contradiction. Finally, we have proved $\nu_{M-1}=0$. \par
Standard Cauchy-Lispchitz arguments, together with the result $p_H(t_{M-1}^-)=p_C(t_{M-1}^-)=0$ yield that $p_H$ and $p_C$ are also identically null on the interval $[t_{M-2},t_{M-1}]$. 
Repeating these arguments on the whole $[0,t_f]$, we find that $p$, $p^0$, $\eta_1$, $\eta_2$ and the $\left(\nu_i\right)_{i=1,\ldots, M}$ are all zero, in contradiction with condition $1$ given by the PMP (see Section \ref{Section3}). 
\end{proof}
Thus $p^0<0$ and we set $p^0=-1$. This normalisation is allowed because the final adjoint vector $(p(t_f),p^0)$ is defined up to scaling.
Again, we start by analysing the PMP on $[t_{M-1},t_f]$.
From $p_C(t_f)<0$, we know that $u_2=u_2^{max}$ and $\phi_1=0$ locally around $t_f$. This implies $p_H>0$ also locally around $t_f$. In particular, $\nu_{M}=0$. Using the same reasoning as before with $p'_C = \beta_C p_C$, we get this time that $p_C$ and $p_H$ have constant sign on $(t_{M-1},t_f)$: $p_C<0$ and $p_H>0$. \par 
Let us now first assume $\nu_{M-1}>0$. Then $\phi_1(t_{M-1}^-)<0$, leading to $u_1=u_1^{max}$ close to $t_{M-1}$. If $\nu_{M-1}$ is such that $p_H \left(t_{M-1}^-\right) \leq 0$, then clearly the maximisation of $\psi(u_2)$ leads to $u_2 = u_2^{max}$. At $t_{M-1}$, we would thus have continuity of $u_2$ and not $u_1$ since $u_1<u_1^{max}$ on $[t_{M-1},t_f]$ from assumption \eqref{Admissible1}. In such a case, it holds true that there can be no jump on the adjoint vector (see for instance~\cite{Bonnard2003}), which is contradictory unless $\nu_{M-1}$ is such that $p_H \left(t_{M-1}^-\right) > 0$, which we assume from now on. \par
Let us now analyse the interval $[t_{M-2} , t_{M-1}]$, on which we will prove that $u_1=u_1^{max}$, $u_2=u_2^{max}$.
Because $\eta_1$ and $\eta_2$ vanish on such an interval, it is easy to prove from standard Cauchy-Lipschitz uniqueness arguments that $p_C<0$ and $p_H>0$ on $[t_{M-2} , t_{M-1}]$. 
Also, because of (\ref{Hyp1}) the inequality 
\begin{equation}
\label{NoSwitch}
\dfrac{\rho_C}{\rho_H} < \frac{\mu_H}{\mu_C} \frac{\mu_C a_{HH} d_H - \mu_H a_{CH} d_C}{\mu_H a_{CC} d_C - \mu_C a_{HC} d_H}
\end{equation}
is satisfied on $[0,t_f]$.
Let us prove that this implies $\phi_1<0$ on $(t_{M-2} , t_{M-1})$. For that purpose, we will prove that whenever $\phi_1(t_0)=0$, its derivative satisfies $\phi'_1(t_0)>0$. Note that we already know that $\phi_1(t_{M-1}^-) \leq \phi_1(t_{M-1}) = 0$.
For such a time $t_0$ we indeed obtain
\begin{align*}
 \phi'_1(t_{0}) = -\frac{\left(p_C \rho_C\right)(t_{0})}{\mu_H}  \Big[ \mu_H\big(  \mu_C a_{HH} d_H & - \mu_H a_{CH} d_C \big) \rho_H(t_{0})  \\
& - \mu_C \big( \mu_H a_{CC} d_C- \mu_C a_{HC} d_H \big) \rho_C(t_{0}) \Big].
\end{align*}
Combined with (\ref{NoSwitch}), this yields $\phi'_1(t_0)>0$, as announced. Thus $u_1=u_1^{max}$ on the whole $[t_{M-2} , t_{M-1}]$. 

For $u_2$, the proof is a bit more involved because the dependence is not linear. In what follows, we generically denote $\phi_{(\lambda_H, \, \lambda_C)} = \lambda_H p_H \rho_H + \lambda_C p_C \rho_C$ for positive constants $\lambda_H$, $\lambda_C$. With this notation the previous established result writes $\phi_{(\mu_H, \, \mu_C)} <0$ on $(t_{M-2} , t_{M-1})$.

We need to maximise $\psi(u_2)=\frac{r_H p_H \rho_H}{1+\alpha_H u_2} +  \frac{r_C p_C \rho_C}{1+\alpha_C u_2}$ as a function of $u_2$, whose derivative has the opposite sign of $P(u_2)$, where
\begin{align*}
P(u) := \alpha_H \alpha_C \phi_{(\alpha_C r_H,\,  \alpha_H r_C)} \, u^2 +2 (\alpha_H \alpha_C) \phi_{(r_H, \, r_C)} \,u
+ \phi_{(\alpha_H r_H, \, \alpha_C r_C)},
\end{align*}
which has discriminant $\Delta = - \alpha_H \alpha_C r_H p_H \rho_H r_C p_C \rho_C (\alpha_C - \alpha_H) ^2 > 0$ on $(0,t_f)$.
We consider two cases, depending on the sign of $\phi_{(\alpha_C r_H, \, \alpha_H r_C)}$. 
Note that (\ref{AssumptionAlpha}) implies the order $\phi_{(\alpha_H r_H, \, \alpha_C r_C)}<\phi_{(r_H, \, r_C)}< \phi_{(\alpha_C r_H, \, \alpha_H r_C)}$. From (\ref{Hyp3}) and $\phi_1<0$, $P(0) = \phi_{(\alpha_H r_H, \, \alpha_C r_C)}<0$.

Let us first assume $\phi_{(\alpha_C r_H, \, \alpha_H r_C)}<0$, in which case all the coefficients of the polynomial are negative. Let us denote $u_{+}$ the greater root of this polynomial. Since the coefficient in front of $u^2$ is negative, the function $\psi$ is increasing with $u_2$ on $(u_{+}, +\infty)$. We cannot have $u_{+}\geq 0$ because of the signs of the coefficients: $u_2^{max}$ maximises the function of interest. If $\phi_{(\alpha_C r_H, \, \alpha_H r_C)}=0$, it is easy to see that the same result holds. \par

%
%
%
%
Now, let us assume that $\phi_{(\alpha_C r_H, \, \alpha_H r_C)}>0$. Because $P(0)<0$, $P(u_2^{max})<0$ is a sufficient condition for $u_2^{max}$ to maximise $\psi(u_2)$. 
For any $\lambda_H>0$, $\lambda_C>0$, $\phi_1<0$ leads to $\phi_{(\lambda_H, \, \lambda_C)} < (\lambda_H \mu_C - \lambda_C \mu_H) \frac{p_H \rho_H}{\mu_C}$. Applying this to $P(u_2^{max})$, we find 
\begin{align*}
P(u_2^{max}) < \frac{p_H \rho_H}{\mu_C} \big( \alpha_H \alpha_C (\alpha _C r_H \mu_C - \alpha_H r_C \mu_H) \, \left(u_2^{max}\right)^2  & \hspace{1.8cm} \\
\hspace{1.8cm} +2 (\alpha_H \alpha_C) (r_H \mu_C - r_C \mu_H)\,u_2^{max}+&(\alpha_H r_H \mu_C - \alpha_C r_C \mu_H)  \big).
\end{align*}
We conclude that $P(u_2^{max})<0$ thanks to (\ref{Hyp4}). 
 
Thus, we have proved that, on $(t_{M-2}, t_{M-1})$, $u_1=u_1^{max}$ and $u_2=u_2^{max}$. Note that the result actually implies $\nu_{M-1} =0$. However the same reasoning with $\nu_{M-1}=0$ works and we obtain $u_1=u_1^{max}$ and $u_2=u_2^{max}$. From assumption (\ref{Decreasing}), $\frac{\rho_C}{\rho_H}$ increases backwards.
If this ratio reaches the value $\gamma$, \textit{i.e.}, if the system saturates the constraint (\ref{cont_HC}) (if not, $t_{M-2} = 0$), then we have a potential boundary arc on (\ref{cont_HC}) on $(t_{M-3}, t_{M-2})$. 
%
%

\textit{Second case: the last arc is a boundary arc on (\ref{cont_H}), reduced to a singleton.}

Note that, again, $t_M = t_f$. This case is handled as the previous one: $p^0$ cannot be $0$ and $\phi_1(t_f^-)\leq 0$. Because of this result, the whole reasoning made above in the previous case applies: there is an unconstrained arc with $u_1=u_1^{max}$ and $u_2^{max}$. If there is a previous arc, it is a constrained arc on (\ref{cont_HC}). 

\textit{Third case: the last arc is a free arc.} 

Again, the same kind of arguments are enough to prove that $p^0 < 0$, and $u_1=u_1^{max}$ and $u_2^{max}$ on this arc.
If there is a previous arc, it is a constrained arc on (\ref{cont_HC}). 
\end{proof}

{\footnotesize
\bibliography{GeneralCitation}
\bibliographystyle{acm}}

\end{document}